\documentclass[preprint,12pt,3p]{elsarticle}

\usepackage{amssymb}
\usepackage{amsmath}
\usepackage{amsthm}
\usepackage{booktabs}
\usepackage{siunitx}
\usepackage{threeparttable}
\usepackage{caption}

\usepackage{thmtools} 
\usepackage{thm-restate}

\newtheorem{claim}{Claim}
\newtheorem{lemma}{Lemma}
\theoremstyle{definition}

\newtheorem{remark}{Remark}

\usepackage{pgfplots}
\pgfplotsset{compat=1.18}

\usepackage[capitalize]{cleveref}
\crefname{claim}{Claim}{Claims}

\newcommand{\argmax}{\arg\!\max} 
\newcommand{\prob}[1]{\mathbf{Pr}\left[#1\right]}

\newcommand{\twostageP}{\textsc{$2$-Stage Reward-DLP}}
\newcommand{\rewardP}{\textsc{Reward-DLP}}
\newcommand{\transitionP}{\textsc{Tran\-sition-DLP}}

\usepackage{xcolor}
\definecolor{KUL_blue}{rgb}{0.11,0.55,0.69}
\newif\ifshowrev
\showrevtrue % \showrevfalse to hide revision

\ifshowrev
  
  \newenvironment{revs}
    {\color{purple}}
    {}
\else

\fi

\usepackage{lineno}

\usepackage{tikz} 
\usetikzlibrary{petri,positioning}

\journal{Discrete Optimization}

\begin{document}

\begin{frontmatter}

%% Title, authors and addresses

%% use the tnoteref command within \title for footnotes;
%% use the tnotetext command for theassociated footnote;
%% use the fnref command within \author or \affiliation for footnotes;
%% use the fntext command for theassociated footnote;
%% use the corref command within \author for corresponding author footnotes;
%% use the cortext command for theassociated footnote;
%% use the ead command for the email address,
%% and the form \ead[url] for the home page:
%% \title{Title\tnoteref{label1}}
%% \tnotetext[label1]{}
%% \author{Name\corref{cor1}\fnref{label2}}
%% \ead{email address}
%% \ead[url]{home page}
%% \fntext[label2]{}
%% \cortext[cor1]{}
%% \affiliation{organization={},
%%             addressline={},
%%             city={},
%%             postcode={},
%%             state={},
%%             country={}}
%% \fntext[label3]{}

\title{Robust Deterministic Policies for Markov Decision Processes under Budgeted Uncertainty}

%% use optional labels to link authors explicitly to addresses:
\author[kul]{Fei Wu}
\author[kul]{Erik Demeulemeester}
\author[kul]{Jannik Matuschke}
\affiliation[kul]{organization={Research Center for Operations Management, KU~Leuven},
            addressline={Naamsestraat~69},
            city={Leuven},
            postcode={3000},
            country={Belgium}}
%%
%% \affiliation[label2]{organization={},
%%             addressline={},
%%             city={},
%%             postcode={},
%%             state={},
%%             country={}}

\author{} %% Author name

%% Author affiliation
% \affiliation{organization={},%Department and Organization
%             addressline={}, 
%             city={},
%             postcode={}, 
%             state={},
%             country={}}

%% Abstract
\begin{abstract}
    This paper studies the computation of robust deterministic policies for Markov Decision Processes (MDPs) in the Lightning Does Not Strike Twice~(LDST) model of \citet*{mannor2012lightning}.
    In this model, designed to provide robustness in the face of uncertain input data while not being overly conservative, transition probabilities and rewards are uncertain and the uncertainty set is constrained by a budget that limits the number of states whose parameters can deviate from their nominal values.
 
    \citet{mannor2012lightning} showed that optimal randomized policies for MDPs in the LDST regime can be efficiently computed when only the rewards are affected by uncertainty.
    In contrast to these findings, we observe that the computation of optimal deterministic policies is $N\!P$-hard even when only a single terminal reward may deviate from its nominal value and the MDP consists of $2$ time periods.
    For this hard special case, we derive a constant-factor approximation algorithm by combining two relaxations based on the \textsc{Knapsack Cover} and the \textsc{Generalized Assignment} problem, respectively.
    For the general problem with possibly a large number of deviations and a longer time horizon, we derive strong inapproximability results for computing robust deterministic policies as well as $\Sigma_2^p$-hardness, indicating that the general problem does not even admit a compact mixed integer programming formulation.
\end{abstract}

%% Keywords
\begin{keyword}
Robust optimization \sep Markov decision processes \sep Approximation algorithms \sep $\Sigma_2^p$-hardness

\end{keyword}

\end{frontmatter}

\section{Introduction}
\label{sec:introduction}
MDPs are models for sequential decision making based on stochastic information about the outcome of individual actions. They have become an important tool for a large variety of applications in diverse fields such as operations research, machine learning, or engineering.
In practice, important parameters of an MDP, such as the rewards or the transition probabilities, are often hard to estimate precisely due to limited and potentially noisy data. 
Not accounting for such uncertainty in the input data when optimizing policies may significantly deteriorate solution quality~\citep{mannor2007bias}. 

This uncertainty motivates the study of \emph{robust} MDP models that optimize policies anticipating a worst-case realization of transition probabilities and rewards from a given uncertainty set.
A widespread assumption in much of the work on such robust MDP models is that the uncertainty set is \emph{rectangular}, i.e., the realization of the uncertain data happens independently at each state of the MDP.
While this assumption makes finding an optimal robust policy computationally tractable in many cases, it may lead to overly conservative solutions, as it allows the worst case to occur at each state simultaneously, a pessimistic scenario that rarely happens in practice.

To mitigate this issue, \citet*{mannor2016robust} introduced the \emph{Lightning Does Not Strike Twice} (LDST) model, which features a budget bounding the number of states for which the parameters may deviate from their given nominal values (within a specified uncertainty set that is independent from deviations at other states). 
The budget thus introduces a global constraint that couples uncertainty realizations at different states, reducing the conservatism of rectangular models.
Moreover, the problem of finding optimal policies remains tractable in the LDST setting, when either allowing the decision maker to employ randomization in the policy and restricting the uncertainty to affect only rewards, or when considering an adaptive model in which both policy decisions as well as uncertainty realizations are history-dependent~\citep{mannor2016robust}. 

Implementing randomized policies, however, is not always possible or desirable in practical contexts. 
Deterministic policies are often easier to implement, verify and communicate~\cite{dolgov2005stationary,nielsen2006finding,ohlmann2009resource}, which is particularly relevant for safety-critical environments where randomness in decision making may raise concerns of accountability and reproducibility.
For example, in health-care applications, randomizing between different treatments may face significant acceptance barriers and would lead to liability issues due to lack of legal frameworks involving randomized decision making~\cite{ayvaci2012budgetary,garcia2023interpretable}. In autonomous driving (and other environments involving human-robotics interaction), randomization may lead to inconsistent behavior under identical external inputs, making the trajectory of a self-driving car hard to predict, e.g., by pedestrians or human drivers of other vehicles, and it may hinder safety certification that requires reproducible behavior.
Similarly, the randomization of automatized systems in industrial processes may lead to irreproducible and hard-to-predict behavior that hinders debugging attempts and may, in hazardous environments, impose serious safety risks during manual interventions.
This issue with unpredictability has, e.g., also been observed in~\cite{ying2023scalable} and~\cite{santana2016rao}.
Finally, in some situations, the MDP may represent a physical network (e.g., a supply chain), in which states correspond to facilities and actions correspond to operation layouts that cannot be easily changed (see, e.g., the production mix example described in \cref{rem:2-stage} or the pavement maintenance problem studied in~\cite{ohlmann2009resource}), making the choice of a fixed layout (i.e., a deterministic action) at each state necessary. For these reasons, system designers may prefer to restrict to deterministic policies even in situations in which randomized policies can (theoretically) perform strictly better in expectation.

In this paper, we therefore investigate the complexity of the LDST model for deterministic policies. We establish that the problem of finding a deterministic policy optimizing the worst-case reward is $\Sigma_2^p$-hard, effectively ruling out the possibility of compact mixed integer programming formulations for the problem under standard complexity assumptions. We moreover observe that even in two-stage MDPs where only a single reward may deviate from its nominal value, the problem remains $N\!P$-hard, contrasting \citeauthor{mannor2016robust}'s \cite{mannor2016robust} positive results for the randomized-policy case. In view of these hardness results, we turn our attention to the computation of approximately optimal policies. While we show that, unless $P = N\!P$, no bounded approximation guarantee can be achieved for the general problem, we derive a constant-factor approximation for the aforementioned two-stage case with reward uncertainty.

\paragraph{Outline of this paper}
In \cref{sec:preliminaries}, we provide an overview of robust MDP models, including a formal description of the general LDST model.
In \cref{sec:conbtributions}, we provide formal statements of our results and discuss their implications, introducing the two-stage special case along the way.
In \cref{sec:reward}, we present two simple reductions that establish our hardness results for the LDST model under reward uncertainty.
In \cref{sec:approximation}, we then describe the constant-factor approximation for two-stage LDST under reward uncertainty, establishing our main positive result.
Finally, in \cref{sec:hardness}, we describe the reductions that prove our complexity results for the general LDST model with uncertainty in the transition probabilities. In \cref{sec: experiments}, we present numerical experiments evaluating the approximation algorithm on several classes of instances and we conclude the paper with open questions in \cref{sec: conclusion}.

\section{Preliminaries: Robust MDP models}
\label{sec:preliminaries}

In this section, we introduce the basic concepts and notation for the (robust) MDP models used in this article, also discussing related results from the literature. We remark that we only cover results directly related to the models studied in the present paper. For a comprehensive overview of the numerous robust-optimization variants of MDP models, we refer the reader to \citep{nilim2005robust, iyengar2005robust, satia1973markovian, white1994markov, givan2000bounded, goyal2023robust}. Before we discuss the robust MDP models leading up to the LDST framework, we first give an introduction to MDPs with known parameters.

\subsection{Markov decision processes with known parameters}

A \emph{Markov decision process} (MDP) consists of a tuple $\langle S, A, p, r, s_0\rangle$ with the following components. The finite set  $ S $ is the \emph{state space}. For each state $s\in S$, there exists a finite set of \emph{actions} $A(s)$. 
The \emph{transition kernel} $p$ consists of a probability distribution $p(\cdot|s, a)$ on $S$ for each $s \in S$ and $a \in A(s)$, with $p(t|s,a)$ denoting the \emph{transition probability} from state $s$ to state $t$ under action $a$.
States $s \in S$ with $A(s) = \emptyset$ are called \emph{terminal states}.
Each terminal state $s$ is equipped with a \emph{reward} $r(s) \in \mathbb{R}$.\footnote{Throughout this paper, we consider only terminal rewards for notational convenience. We remark that in the context of our study, this is without loss of generality compared to the more common setting where rewards are associated with actions. We discuss this in more detail in \cref{rem:wlog} in \cref{sec:conbtributions}.}
Finally, there exists an \emph{initial state} $s_0 \in S$.
Throughout the paper, we will only consider \emph{finite-horizon} MDPs, that is, we will assume that $S$ is partitioned in sets $S_0, \dots, S_{T}$ for some $T \in \mathbb{N}$ such that $S_0 = \{s_0\}$, the states in $S_T$ are exactly the terminal states, and $p(t|s,a) > 0$ for $s, t \in S$ and $a \in A(s)$ implies $s \in S_{\tau}$ and $t \in S_{\tau+1}$ for some $\tau \in \{0, \dots, T-1\}$.

A (deterministic, history-independent) policy is a function $\pi$ that assigns to each non-terminal state $s \in S$ an action $\pi(s) \in A(s)$.
Each policy defines a random sequence of states $s^{\pi}_0, \dots, s^{\pi}_T$ with $\prob{s^{\pi}_0 = s_0} = 1$ and $\prob{s^{\pi}_{\tau+1} = t | s^{\pi}_{\tau} = s} = p(t | s, \pi(s))$ for $\tau \in \{0, \dots, T-1\}$ and $s \in S_{\tau}$.
The \emph{(nominal) reward} obtained by policy $\pi$ is 
\begin{align*}
    R(\pi) := \mathbb{E}\left[r(s^{\pi}_T)\right]=\sum_{s \in S_{T}} \prob{s^{\pi}_T = s} r(s).
\end{align*}
It is well known that a policy achieving a maximum reward can be computed in polynomial time, e.g., using dynamic programming or a linear programming formulation.
It is also well-known that the maximum reward that can be achieved by a deterministic, history-independent policy as defined here is the same as the maximum reward that can be achieved by more general randomized, history-dependent policies where the action taken by the policy may be random and/or depend on previous states of the sequence \citep{puterman1994markov}.

\subsection{Robust MDPs: The general framework and rectangularity}
As mentioned in \cref{sec:introduction}, parameters such as the transition probabilities~$p$ and the rewards $r$ are often hard to estimate.
To address this issue of parameter uncertainty, MDPs have been studied via the lens of (distributionally) robust optimization~\citep{nilim2005robust, iyengar2005robust, satia1973markovian, white1994markov, givan2000bounded}. 
In the general robust MDP framework, we are given an MDP $\langle S, A, p, r, s_0\rangle$ and an \emph{uncertainty set} $\hat{\mathcal{P}}$ of alternate transition kernels $\hat{p}$, i.e., $\hat{p}(\cdot|s,a)$ is a probability distribution on $S$ for each $\hat{p} \in \hat{\mathcal{P}}$ and each $s \in S$ and $a \in A(s)$.
The goal is to find a policy maximizing the \emph{worst-case reward} among all possible transition kernels from~$\hat{\mathcal{P}}$, that is,
\begin{align*}
    \hat{R}(\pi) := \min_{\hat{p} \in \hat{\mathcal{P}}} \mathbb{E}\left[r(s^{\pi,\hat{p}}_T)\right]= \min_{\hat{p} \in \hat{\mathcal{P}}} \sum_{s \in S_{T}} \prob{s^{\pi,\hat{p}}_T = s} r(s),
\end{align*}
where $s^{\pi,\hat{p}}_0, \dots, s^{\pi,\hat{p}}_T$ is the random sequence of states induced by policy $\pi$ when using transition kernel $\hat{p}$ instead of $p$.

The complexity of finding optimal policies in the robust MDP framework can vary greatly, depending on the nature of the uncertainty sets.
An important class of uncertainty sets for which computing optimal robust policies remains tractable are so-called \emph{$(s, a)$-rectangular} uncertainty sets~\citep{nilim2005robust,iyengar2005robust}.
In such uncertainty sets, the realization of the transition probabilities can be chosen independently for each state-action pair $s \in S$ and $a \in A(s)$ from a local uncertainty set $\hat{P}_{s, a}$ containing probability distributions over $S$, i.e., an $(s, a)$-rectangular uncertainty set is of the form $\hat{\mathcal{P}} = \times_{s \in S, a \in A(s)} \hat{{P}}_{s,a}$. \citet{wiesemann2013robust} extend this concept to a general class called $s$-rectangular uncertainty sets, where the transition probabilities can be selected independently for each state $s \in S$ from the uncertainty set $\hat{P}_{s}=\{\hat{{P}}_{s,a}\;:\; a\in A(s)\}$, and transition probabilities for different actions in the same state may be dependent, i.e., $\hat{\mathcal{P}} = \times_{s \in S} \hat{{P}}_{s}$.  While for $(s, a)$-rectangular uncertainty sets it remains true that allowing randomized policies does not improve the optimal worst-case reward, \citet{wiesemann2013robust} observe that this is no longer true for $s$-rectangular uncertainty sets. They show that an optimal randomized policy can be computed in polynomial time for this type of uncertainty set by a robust value iteration approach.

Some authors establish the optimality of deterministic policies also for $s$-rectangular uncertainty sets under additional conditions. \citet{li2025rectangularity} show that policy optimization for certain robust MDPs can be reformulated as a dynamic game and that for those MDPs the existence of  deterministic optimal policies can be characterized via saddle points in the game. They show that this characterization can, in particular, be applied when the uncertainty set is \emph{$sr$-rectangular}, a generalization of $s$-rectangularity.
\citet{wang2024reliable} combine the idea of $s$-rectangularity with a Wasserstein metric (a widely used metric in distributionally robust optimization, stemming from optimal transport) to model uncertainty sets constructed by observing a given behavioral policy.
They show that in their framework, under certain technical conditions, optimal deterministic policies exist and can be computed efficiently by policy iteration.

\subsection{The LDST model}
While rectangular uncertainty sets allow for an efficient computation of robust policies, the resulting robust MDP models are based on the fairly conservative assumption that the worst-case realization can occur simultaneously at each action chosen by the policy at hand.
To mitigate the conservatism of $(s, a)$-rectangular uncertainty sets, \citet{mannor2016robust}, taking inspiration in the widely used budgeted uncertainty framework of \citet{bertsimas2004price}, introduced  the LDST model. 
The LDST model allows to adjust the desired degree of robustness/conservatism via an \emph{uncertainty budget} $k \in \mathbb{N}$ that determines the number of state-action pairs $s \in S$ and $a \in A(s)$ for which the transition probabilities may deviate from their specified nominal values~$p(\cdot | s, a)$.
Formally, a \emph{budgeted uncertainty set} in the LDST model is of the form
\begin{align*}
P = \bigl\{ \hat{p} \;:\; & \hat{p}(\cdot | s, a) \in \hat{\mathcal{P}}_{sa} \text{ for all } s \in S, a \in A(s) \text{ and } \\
& |\{(s, a): s \in S, a \in A(s),\, \hat{p}(\cdot | s, a) \neq p(\cdot | s, a)\}| \leq k \bigr\},
\end{align*}

where each $P_{sa}$ for $s \in S$ and $a \in A(s)$ is a set of distributions over $S$ with $p(\cdot|s, a) \in \hat{P}_{sa}$.

\citet[Theorem~2]{mannor2012lightning} observed that even evaluating the worst-case reward $\hat{R}(\pi)$ of a given policy $\pi$ is $N\!P$-hard in the LDST regime with uncertain transition kernels, no matter whether the policy in question is randomized or deterministic, Markovian or history-dependent.
They therefore investigated the LDST model under \emph{uncertain rewards}, where all transition probabilities retain their nominal values but the rewards of $k$ terminal states may deviate from their nominal values within given uncertainty sets, that is, the goal is to find a policy maximizing the worst-case reward
\begin{align*}
    \hat{R}(\pi) := \min_{\hat{r} \in \hat{\mathcal{Q}}} \mathbb{E}\left[\hat{r}(s^{\pi}_T)\right]= \min_{\hat{r} \in \hat{\mathcal{Q}}} \sum_{s \in S_{T}} \prob{s^{\pi}_T = s} \hat{r}(s),
\end{align*}
where the uncertainty set $\mathcal{Q}$ is of the form
\begin{align*}
    \hat{\mathcal{Q}} = \big\{\hat{r} \in \mathbb{R}^{S_{T}} \;:\; & \hat{r}(s) \in \{r(s), r'(s)\} \text{ for all } s \in S_T \text{ and } \\
    &  |\{s \in S_T :\, \hat{r}(s) \neq r(s)\}| \,\leq\, k\big\},
\end{align*}
where $r'(s) \leq r(s)$ is the alternative reward at terminal state $s \in S_T$ within the uncertainty set.
Note that reward uncertainty can easily be modeled using uncertainty in the transition probabilities by a standard reduction, so reward uncertainty constitutes a special case of transition uncertainty.
\citet[Theorem~3]{mannor2012lightning} showed that in this special case, optimal randomized policies can be computed efficiently by applying a separation oracle to the standard LP formulation for policy optimization in finite-horizon MDPs.
Note however that this positive result does not automatically carry over to the computation of deterministic policies, because the optimal worst-case rewards that can be obtained by randomized or deterministic policies, respectively, can differ by an arbitrary factor in the LDST model, as can be seen in the example given in  \cref{fig:example-randomized}. There, the optimal randomized policy takes action $a$ and $b$, respectively, each with probability~$\frac{1}{2}$, achieving an expected reward of~$\frac{1}{2}$. However, the optimal value achievable by any deterministic policy is~$0$.
\begin{figure}
\centering
\begin{tikzpicture}[thick, every transition/.style={draw=red,fill=red!20,minimum size=4.8mm}, every place/.style={draw=blue,fill=blue!20,minimum size=6mm}]
\node[place] (s0) {$s_0$};
\path (s0) ++(1.5, 0.5) node[transition] (a) {$a$} edge [pre] (s0); 
\path (s0) ++(1.5, -0.5) node[transition] (b) {$b$} edge [pre] (s0); 
\node[place, right= 1.5cm of a, label={right: $ 1/ \textcolor{red}0 $}] (t1) {$t_1$} edge  [pre] node[pos=0.5, fill=white]{$1$} (a);
\node[place, right= 1.5cm of b, label={right: $ 1/ \textcolor{red}0 $ }] (t2) {$t_2$}  edge [pre]node[pos=0.5,fill=white]{$1$} (b);
\end{tikzpicture}
\caption{\footnotesize{Example showing that, in the LDST model, the reward of a randomized policy can exceed that of an optimal deterministic policy. Each round node represents a state, each square represent an action. Numbers on the arcs denote the (nominal) transition probabilities. Numbers next to the terminal states indicate nominal (first number, black) and worst-case (second number, red) reward, respectively.}}
\label{fig:example-randomized}
\end{figure}

\citet{mannor2012lightning} also consider an \emph{adaptive} version of the LDST model, in which both the policy and the uncertainty realization may depend on the previous history, in particular on the realizations of the random outcomes of actions taken by the policy, and show that optimal policies can be efficiently computed via dynamic programming in this case, a result that is later extended to a more general setting called $k$-rectangular uncertainty sets~\citep{mannor2016robust}.

\section{Contributions of this article}
\label{sec:conbtributions}

We study the computation of optimal deterministic policies in the LDST model.
We first consider the case of reward uncertainty. 
To this end, we define the problem {\rewardP} (where DLP stands for \textsc{Deterministic LDST Policy}): Given an MDP $\langle S, A, p, r, s_0\rangle$ and a budgeted uncertainty set $\hat{\mathcal{Q}}$ for the terminal rewards, our goal is to find a (deterministic and history-independent) policy $\pi$ maximizing the worst-case reward~$\hat{R}(\pi)$.

An interesting special case of this problem is the two-stage case, where $T=2$, i.e., the state space can be partitioned into the stages $S_0 = \{s_0\}, S_1, S_2$, with $S_2$ consisting of terminal states, and where the uncertainty budget $k$ equals $1$, i.e., only a single terminal reward can deviate from its nominal value in any scenario.
We will refer to this special case of {\rewardP} as {\twostageP}.
We first observe that even this special case is strongly $N\!P$-hard.

\begin{restatable}{theorem}{thmTwoStageHardness}
\label{thm:2-stage-hardness}
    {\twostageP} is strongly $N\!P$-hard.
\end{restatable}

\cref{thm:2-stage-hardness} follows from a simple reduction from \textsc{3-Partition} that is described in \cref{sec:2-stage-hardness}.
The theorem reveals a striking contrast to the randomized-policy setting, for which optimal policies under reward uncertainty can be computed in polynomial time, even when the number of reward parameters differing from their nominal value is part of the input~\citep[Theorem~3]{mannor2012lightning}.
Our earlier example given in \cref{fig:example-randomized} further reveals that the randomized-policy setting is also not suited as an approximation for optimal deterministic policies.
This raises the question whether {\twostageP} or even {\rewardP} in general allows for an approximation algorithm.
An \emph{$\alpha$-approximation algorithm} for $\alpha \in (0, 1]$ is a polynomial-time algorithm that computes a policy $\pi$ with $\hat{R}(\pi) \geq \alpha \hat{R}(\pi^*)$, where $\pi^*$ is a policy with maximum worst-case reward for the given instance.
Unfortunately, the following theorem answers this question in the negative for {\rewardP}.

\begin{restatable}{theorem}{thmRewardInapproximability}
\label{thm:reward-inapproximability}
    Unless $P = N\!P$, there does not exist an $\alpha$-approximation algorithm for {\rewardP} for any constant $\alpha \in (0, 1]$. 
\end{restatable}

The proof of \cref{thm:reward-inapproximability}, which is given in \cref{sec:reward-inapproximability}, reveals an intrinsic connection between the \textsc{Disjoint Paths} problem and {\rewardP}.

Given the inapproximability result for the general version of {\rewardP}, we turn our attention again to {\twostageP}.
Intuitively, {\twostageP} can be seen as a load-balancing problem in which the goal is to maximize the total reward at the terminal states while keeping the reward that can be lost at any terminal state in the worst case small.
As such, {\twostageP} is interesting in its own right.
We derive the following result, showing that {\twostageP} is more tractable than general {\rewardP} in the sense that polynomial-time algorithms can at least recover a constant fraction of the optimal reward.

\begin{restatable}{theorem}{thmApproximation}
\label{thm:approximation}
	For any $\varepsilon > 0$, there exists a $\frac{1}{5 + \varepsilon}$-approximation algorithm for {\twostageP}.
\end{restatable}

The algorithm underlying \cref{thm:approximation} is based on a combination of two carefully chosen relaxations that complement each other, with different performances depending on the fraction of nominal reward that is lost in the worst-case scenario. We devote \cref{sec:approximation} to the description of the algorithm.

\cref{thm:approximation} raises the question whether more general versions of the problem can be approximated, e.g., {\rewardP} with two stages but with an arbitrary uncertainty budget.
We observe that this variant of {\rewardP} generalizes the \textsc{Santa Claus} problem, for which finding a constant-factor approximation is a long-standing open question~\citep{bansal2006santa}.

\begin{restatable}{theorem}{thmSanta}
\label{thm:santa}
	If there exists an $\alpha$-approximation for {\rewardP} restricted to instances with $T = 2$, then there exists an $\alpha$-approximation for \textsc{Santa Claus}.
\end{restatable}

Finally, we turn our attention to the complexity of computing deterministic policies in the more general case of LDST with an uncertain transition kernel.
Formally, we consider the problem {\transitionP}, in which we are given an MDP $\langle S, A, p, r, s_0\rangle$ and a budgeted uncertainty set $\hat{\mathcal{P}}$ for the transition kernel, and the goal is to find a policy $\pi$ maximizing the worst-case reward~$\hat{R}(\pi)$.
We show that uncertain transition kernels are indeed significantly harder to handle than uncertain rewards, by proving the following two theorems on the complexity of {\transitionP}.

\begin{restatable}{theorem}{thmTransitionInapproximability}\label{thm:transition-inapproximability}
	The following problem is $N\!P$-hard: Given an instance of {\transitionP} with uncertainty budget $k = 2$, decide whether there exists a policy $\pi$ with $\hat{R}(\pi) > 0$.
\end{restatable}

\begin{restatable}{theorem}{thmTransitionSigma}\label{thm:transition-sigma-2-p}
	{\transitionP} is $\Sigma_{2}^{p}$-hard. 
\end{restatable}

The proofs of \cref{thm:transition-inapproximability,thm:transition-sigma-2-p} are discussed in \cref{sec:transition-inapproximability,sec:transition-sigma-2-p}.
Note that \cref{thm:transition-inapproximability} implies that (unless $P = N\!P$) there does not exist a polynomial-time approximation algorithm for {\transitionP}, even when only two state-action pairs can be affected by the uncertainty at any given time.
Moreover, the $\Sigma_{2}^{p}$-hardness postulated in \cref{thm:transition-sigma-2-p} implies that we cannot even hope to formulate {\transitionP} as an integer program of polynomial size~\citep{woeginger2021trouble}, unless the polynomial hierarchy collapses to its second level.

\paragraph{Discussion of modelling assumptions}

Before we turn to the proofs of the results mentioned above, we remark that some of the modelling assumptions we made are without loss of generality in the context of these results and provide motivation for the two-stage model as an important special case.

\begin{remark}[Uncertainty in terminal rewards]
\label{rem:wlog}
    The assumption that rewards are only obtained at the terminal states rather than associated with each action (or state-action pair) is without loss of generality:
    In {\twostageP}, we can model rewards at actions by introducing for each action $a \in A(s)$ for $s \in S_1$ an additional terminal state $s_a$ with $r(s_a) = \frac{r(a)}{\varepsilon}$ for some fixed $\varepsilon > 0$ and replacing the transition probabilities for $a$ by $p'(s_a|s,a) = \varepsilon$ and $p'(s'|s,a) = (1 - \varepsilon)p(s'|s,a)$ for all $s' \in S_2$; we further scale up all original terminal rewards (if present) by $\frac{1}{1 - \varepsilon}$. 
    For rewards of an action $a \in A(s_0)$, we likewise introduce a terminal state $s_a$ with $r(s_a) = \frac{r(a)}{\varepsilon}$ to $S_2$ and an intermediate state $s'_a$ to $S_1$ such that $s'_a$ has a single action deterministically leading to $s_a$ and change the transition probabilities at $a$ to $p'(s'_a|s_0,a) = \varepsilon$ and $p'(s'|s_0,a) = (1 - \varepsilon)p(s'|s_0,a)$ for all $s'\in S_1$; we further scale up all original rewards and rewards corresponding to actions in $S_1$ by $\frac{1}{1 - \varepsilon}$.
    The resulting instance with terminal rewards behaves identically to the original instance with rewards at actions.
    Conversely, terminal rewards can be easily modeled as rewards by actions by introducing a single action at each terminal state action that yields the reward and leads to a new artificial terminal state.
    Hence our complexity results for {\rewardP} and {\transitionP} are also valid for the case of reward at actions.
\end{remark}

\begin{remark}[Finite horizon]
    While we only consider finite-horizon MDPs in this paper, our complexity results for {\rewardP} and {\transitionP} generalize directly to the infinite horizon case, where the terminal states simply become absorbing states.
\end{remark}

\begin{remark}[History-independent policies]
	Throughout the paper, we consider history-independent policies. Note, however, that many of our results pertain to a two-stage setting, where restricting to history-independent policies is without loss of generality: in the two-stage setting any intermediate state is necessarily preceded by the initial state $s_0$ and therefore any policy is history-independent. 
	Thus, our results for the two-stage setting (\cref{thm:2-stage-hardness,thm:approximation,thm:santa}), both in terms of approximation and in terms of hardness, still hold true when allowing history-dependent policies. Indeed, note that hardness for the two-stage setting also implies hardness for larger time horizons, as the MDPs constructed in the reductions can be extended by an arbitrary number of trivial stages that do not affect history-dependent policies. Thus, \cref{thm:2-stage-hardness,thm:santa} imply that computing optimal history-dependent deterministic policies is $N\!P$-hard for arbitrary time horizons and that approximating it by a constant factor is at least as hard as approximating the \textsc{Santa Claus} problem.
We point out, however, that history-dependent decisions can yield higher rewards for larger time horizons in general.
In fact, similarly to deterministic vs.~randomized policies, the worst-case gap between an optimal history-dependent vs.~an optimal history-independent policy cannot be bounded by a constant; see \cref{rem:history-gap} in \cref{sec:reward-inapproximability}.
In view of this gap, identifying cases in which (near-)optimal history-dependent deterministic policies can be computed efficiently is an interesting question for future research.
That said, we also point out that history-independence may be strongly desirable in some situations: e.g., for some automated systems (such as self-driving cars) where humans interacting with the system (such as pedestrians) can only observe the current state of the system but not the history, a history-dependent policy may make it impossible to reliably anticipate the system's actions, causing potential safety risks.
Furthermore, in some robotics applications, history-dependent policies are not feasibly implementable due to hardware limitations, as they require more complex local controller components and possibly extensive memory~\cite{dolgov2005stationary,bonet2009automatic}.
\end{remark}

\begin{remark}[Two-stage model]\label{rem:2-stage}
In several of our results we consider the problem with only two stages.
While this might seem restrictive, {\rewardP} with two stages already generalizes several classic optimization problems, including the \textsc{Santa Claus} problem (as shown in \cref{thm:santa}). Indeed, {\twostageP} is very closely related to the generalized load balancing~\cite{kesselheim2023online} and configuration balancing~\cite{eberle2025configuration} problems, which both ask for choosing configurations for different requests (the states in $S_1$), with each configuration incurring a different load profile on the underlying resources (the state in $S_2$), the difference lying only in the objective function: while configuration balancing and generalized load balancing strive for minimizing a norm of the load profile, {\twostageP} aims at maximizing the reward (which is a weighted $\lVert\cdot\rVert_1$-norm of the load profile) under a worst-case deviation scenario for the rewards coefficients. 
In the following, we want to discuss another application of the two-stage case, showing that it captures the essence of production mix and routing problems that have been extensively studied in literature~\cite{zhang2017multiscale,long2024data,oliveira2018uncertainty}. For this, consider the following production-mix problem:
We are given a set of facilities, each of which can operate in a number of different modes, each mode resulting in a different product mix as output for the facility. Let $\alpha_{imj}$ indicate the amount of product $j$ at facility $i$ when it is in mode $m$. By scaling, we can assume $\sum_{j} \alpha_{imj} \leq 1$ for all $i$ and $m$.
The market value of each product is uncertain, having a nominal value and a lower alternative value.
The goal is to choose production modes for the facilities so as to maximize the revenue in the worst-case realization of product values, assuming that at most $k$ products deviate from their nominal value.
This problem can be modeled as an instance of {\rewardP} with $T = 2$ as follows. In $S_0$ we include a starting state with a single action that with equal probabilities leads to one of the states in $S_1$, each of which represents a facility.
The set of terminal states consists of a state $s_j$ for each product $j$ and an additional state $\bar{s}$.
We set $r(s_j)$ and $r'(s_j)$, respectively, to the nominal and deviating value of product $j$, respectively and set $r(\bar{s}) = r'(\bar{s}) = 0$.
Each action $a_m$ at sate $s_i \in S_1$ corresponds to a possible production mode $m$ for facility $i$, with transition probability $p(s_j \mid s_i, a_m) = \alpha_{imj}$.
We further set $p(\bar{s} \mid s_i, a_m) = 1 - \sum_{j} \alpha_{imj}$ to ensure that the total probability for action $a_m$ equals $1$.
It is easy to see that deterministic policies to the constructed MDP are in one-to-one correspondence to solutions to the product mix problem, with the worst-case reward achieved being equal to the revenue under the worst-case realization of product values for that solution multiplied by $1 / |S_1|$.
Our approximation result for {\twostageP}, which corresponds to {\rewardP} with $T = 2$ and adversarial budget $k = 1$, thus implies an identical approximation result for the product-mix problem with $k = 1$.
\end{remark}

\paragraph{Comparison with previous complexity results for non-rectangular uncertainty sets} Several authors have previously studied the problem of evaluating MDP policies under non-rectangular uncertainty sets. In particular, \citet{wiesemann2013robust} consider convex (but otherwise arbitrary) uncertainty sets for transition kernels and show that determining the worst-case realization for a given policy is already strongly $N\!P$-hard in this setting. They achieve this by encoding an arbitrary binary integer program into the description of the uncertainty set, where the integrality constraints of the integer program are enforced by a suitably chosen MDP structure.
Similarly, \citet{goh2018data} study the computation of the worst-case discounted net present value of a Markov chain whose transition matrix is uncertain but confined to a convex uncertainty set. They show $N\!P$-hardness of this evaluation problem for the case that the transition kernel has to correspond to a doubly stochastic matrix with some transitions forced to $0$. 
This result is achieved via a reduction from \textsc{Hamiltonian Cycle}, in which the long cycle is used to reduce the frequency of visiting a particular state in the infinite horizon setting.
Finally, as mentioned earlier, \citet{mannor2012lightning} establish $N\!P$-hardness for evaluating policies in the LDST model with uncertain transition kernels, using a reduction from \textsc{Vertex Cover}.

We remark that none of our results is implied by any of the aforementioned results.
In particular, for the LDST model with reward uncertainty, which is the setting for most of our results, the aforementioned hardness results for the evaluation of policies do not hold: different from the uncertain transition kernels, deviations in rewards at different state-action pairs do not mutually influence each other and therefore a worst-case realization of rewards for a given policy can be easily constructed greedily in polynomial time.
To prove our various hardness results for optimizing policies under uncertain rewards, we therefore devise entirely new constructions (based on scheduling/load balancing problems as well as a disjoint path problem). 
Our results for uncertain transition kernels can be seen as strengthenings of the results by \citet{mannor2012lightning} in two directions: \cref{thm:transition-inapproximability} establishes inapproximability (even for adversarial budget $k=2$, where evaluation is again trivial); \cref{thm:transition-sigma-2-p} establishes hardness at a higher level of the polynomial hierarchy.
Indeed, our proof of \cref{thm:transition-sigma-2-p} uses a reduction from \textsc{Max-Min Vertex Cover}, which is a bilevel version of the \textsc{Vertex Cover} problem used in the proof of \citet[Theorem~2]{mannor2012lightning}.

\section{Complexity for LDST with reward uncertainty}
\label{sec:reward}

In this section, we describe the reductions that prove $N\!P$-hardness for {\twostageP} (\cref{thm:2-stage-hardness}), the inapproximability of {\rewardP} (\cref{thm:reward-inapproximability}), and the fact that \textsc{Santa Claus} is a special case of {\rewardP} with $T = 2$ (\cref{thm:santa}).
We start by proving \cref{thm:2-stage-hardness}, which we restate here for convenience.

%\subsection{$N\!P$-hardness of {\twostageP}}
\label{sec:2-stage-hardness}

\thmTwoStageHardness*

\begin{figure}
		\centering 
		\scalebox{0.8}{\begin{tikzpicture}[yscale=-1.25,xscale=1.2,thick,
			every transition/.style={draw=red,fill=red!20,minimum size=4.8mm},
			every place/.style={draw=blue,fill=blue!20,minimum size=6mm}]
			\node[place] (s0)at (-1,0) {$s_0$};
			\node[place] (s1) at (2,-1.5) {$s_1$};
			\node[place, label=right:{$ 1/ \textcolor{red}0 $}] (t1) at (6,-1.5) {$t_1$};
			\foreach \i in {2} {
				\node[place] (s\i) at (2,-3+1.5*\i) {$s_\i$};
				\node[place, label=right:{$ 1/ \textcolor{red}0 $}] (t\i) at (6,-3+1.5*\i) {$t_\i$};
			}
			%s1\label=above:{$t=1$}];
			\node[place] (s{3n}) at (2,1.6) {$s_{3n}$};
			\node[place,label=right:{$ 1/ \textcolor{red}0 $}] (t{n}) at (6,1.6) {$t_{n}$};
			\path (s2) --++ (0,1.2) node[midway,scale=1.5] {$\vdots$};
			\path (t2) --++ (0,1.2) node[midway,scale=1.5] {$\vdots$};
			\node[transition] at (0.5,0) {$ a_0 $}  edge [pre] (s0) edge [post] node[fill=white]{$\frac{b_1}{nB}$} (s1) edge [post] node[fill=white]{$\frac{b_2}{nB}$} (s2) edge [post] node[fill=white]{$\frac{b_{3n}}{nB}$} (s{3n}); 
			% insert m actions for each state
			% s_1
			\foreach \i in {1,2} {
				\node[transition] (m\i) at (3.5,-2.5+0.5*\i) {$a_\i$}  edge [pre] (s1)   edge [post]   (t\i);
			}
			\node[transition] (m3) at (3.5,-0.6) {$a_n$} edge  [pre] (s1)  edge [post]  (t{n});
			
			\path (m2) --++ (0,0.5) node[midway,scale=1.5] {$\vdots$};
			\path (m3) --++ (0,1.2) node[midway,scale=1.5] {$\vdots$};
			% s_3n
			\foreach \i in {1} {
				\node[transition] (m3\i) at (3.5, 0.6+0.5*\i) {$a_\i$} edge [pre] (s{3n})  edge [post]  (t\i);
			}
		\foreach \i in {2} {
			\node[transition] (m3\i) at (3.5, 0.6+0.5*\i) {$a_\i$} edge [pre] (s{3n})  edge [post]  (t\i);
		}
			\node[transition] (m33) at (3.5,2.5) {$a_n$} edge [pre] (s{3n})  edge [post]  (t{n}); 
			\path (m32) --++ (0,0.5) node[midway,scale=1.5] {$\vdots$};
			\foreach \i in {1,2} {
				\node[transition] (a{\i}) at (3.5,-2.5+0.5*\i) {$a_\i$};
			}
			%construct transfer line
		\end{tikzpicture}}
		\caption{Construction of the reduction from \textsc{3-Partition} to {\twostageP} in the proof of \cref{thm:2-stage-hardness}.}
    \label{fig:2-stage-hardness}
\end{figure}

\begin{proof}
    We prove the theorem by reduction from \textsc{3-Partition} \citep{garey1975complexity}.
    Thus, we are given $b_1, \dots, b_{3n} \in \mathbb{N}$ and $B \in \mathbb{N}$ with $\sum_{i=1}^{3n} b_i	= nB$, and our goal is to decide whether there is a partition $U_1, \dots, U_n$ of $U:=\{1, \dots, 3n\}$ with $ \sum_{j \in U_i} b_j = B$ for all $i \in N:= \{1, \dots, n\}$.
    We construct an instance of {\twostageP} as follows.
    The state space is $S = S_0 \cup S_1 \cup S_2$ with 
    \begin{align*}
        S_0 := \{s_0\}, \quad S_1 := \{s_i:\; i\in U\}, \quad S_3 := \{t_j:\;j\in N\}.
    \end{align*}
    The sets of actions are
        $A(s_0) := \{a_0\} \text{ and } A(s_i) := \{a_j:\; j\in N\}  \text{ for } i \in U$.
    The transition probabilities are given by
    $p(s_i|s_0, a_0) := \frac{b_i}{nB}$ for $i \in U$
    and $p(t_j|s_{i}, a_j) := 1, \text{ for } i\in U, ~j\in N$. 
    All other transition probabilities are $0$.
    For each $j \in N$, the nominal reward at the corresponding terminal state is $r(t_j) := 1$ and the alternative reward $\hat{r}(t_j) := 0$.
    See \cref{fig:2-stage-hardness} for a depiction of the entire construction.
    We show that there is a partition $U_1, \dots, U_n$ of $U$ with $ \sum_{j \in U_i} b_j = B$ for all $i \in N$ if and only if there is a policy $\pi$ with $\hat{R}(\pi) = 1 - \frac{1}{n}$.
    
    Note that any policy in the constructed instance is completely determined by the actions it takes for the states in $S_1$.
    Hence any policy $\pi$ induces a partition of $U$ given by $U^{\pi}_j := \{i \in U \,:\, \pi(s_i) = a_j\}$ for $j \in N$.
    Conversely, for every partition of $U$ into $n$ sets $U_1, \dots U_n$, there is a policy $\pi$ such that $U_j=U^{\pi}_j$ for all $j \in N$.
    It thus suffices to show that $\hat{R}(\pi) = 1 - \frac{1}{n}$ for policy $\pi$ if and only if $\sum_{i\in U^{\pi}_j} b_i = B$ for all $j \in N$.
    
    For some policy $\pi$ and $j \in N$ let
    \begin{align} 
        p^{\pi}(t_j) := \prob{s_T^{\pi} = t_j} =\sum_{i\in U : \pi(s_i) = a_j} \frac{b_i}{nB} = \sum_{i\in U^{\pi}_j} \frac{b_i}{nB},\label{eq:partition-value}
    \end{align}
    denote the probability that the process ends at state $t_j$.
	Note that 
    \begin{align*}
        \hat{R}(\pi) = \min_{j' \in N} \sum_{j \in N \setminus \{j'\}} r(t_j) p^{\pi}(t_j) = \min_{j' \in N}  1 - p^{\pi}(t_{j'}) = 1 - \max_{j' \in N} p(t_{j'}).
    \end{align*}
    Note further that $\max_{j' \in N} p^{\pi}(t_{j'}) \geq \frac{1}{n} \sum_{j \in N} p^{\pi}(t_j) = \frac{1}{n}$, with equality if and only if $p^{\pi}(t_{j'}) = \frac{1}{n}$ for all $j \in N$.
    By \eqref{eq:partition-value}, this is the case if and only if $\sum_{i\in U^{\pi}_j} b_i = B$ for all $j \in N$. 
    We conclude that $\hat{R}(\pi) = 1 - \frac{1}{n}$ if and only if $\sum_{i\in U^{\pi}_j} b_i = B$ for all $j \in N$.
\end{proof}

\label{sec:reward-inapproximability}

We now turn our attention to the proof of \cref{thm:reward-inapproximability}, which shows inapproximability of the general version of {\rewardP}.

\thmRewardInapproximability*
\begin{proof}
    By contradiction, assume that there exists an $\alpha$-approximation for {\rewardP} for some $\alpha \in (0, 1]$.
    We show that this implies $P = N\!P$ by reduction from \textsc{Vertex-Disjoint Paths}.
    
    Given a directed graph $ D = (V, E) $ and a set of $\ell$ pairs of vertices $(s_1, t_1), \dots, (s_{\ell},t_{\ell})$, the \textsc{Vertex-Disjoint Paths} problem asks whether there exists a set of vertex-disjoint paths $P_1, \dots, P_{\ell}$ such that every $P_i$ is an $s_i$-$t_i$-path for all $i \in L := \{1, \dots, \ell\}$.
When $\ell$ is part of the input, this problem is NP-hard, even when the underlying digraph $D$ does not contain directed cycles~\citep{even1976complexity}.
We thus assume that $D$ is a directed acyclic graph (DAG).

We further assume that $D$ has a layered structure where the set of vertices $V$ is partitioned into $V_1, \dots, V_{T}$ for some $T \in \mathbb{N}$ such that $V_1 = \{s_1, \dots, s_{\ell}\}$, $V_{T} = \{t_1, \dots, t_{\ell}\}$ and $e = (u, v) \in E$ implies $u \in V_{\tau}$ and $v \in V_{\tau+1}$ for some $\tau \in \{1, \dots, T\}$.
This assumption is without loss of generality as $D$ is a DAG, so we can partition the original vertex set according to a topological ordering and then subdivide all edges to ensure they do not skip any layers.

We construct an instance of {\rewardP} as follows. 
The state space is $S = V \cup \{s_0\}$.
The action sets are 
 $A(s_0) = \{a_0\}$ and $A(v)=\{a_e:\; e\in \delta^+(v)\}$ for $v\in V \setminus \{t_1,\dots,t_k\}$.
Let $0 < \varepsilon^i < \alpha$ and $\delta = \sum_{i=1}^{\ell} \varepsilon^i$. The transition probabilities for action $a_0$ are given by $p(s_i|s_0,a_0) = \frac{\varepsilon^{i}}{\delta}$ for $i \in L$.
Each action $a_e$ for $e = (v, w) \in A$ deterministically leads from state $v$ to state $w$, i.e., $p(w | v, a_e) = 1$.
The nominal and worst-case rewards, respectively, for the terminal states $t_1, \dots, t_{\ell}$ are given by $r(t_i) = \frac{\delta}{\varepsilon^i}$ and $r'(t_i) = 0$, respectively, for $i \in L$.
The number of states for which the reward can deviate is given by $k = \ell - 1$. An example of the construction is visualized in \cref{k-DVP}.
\begin{figure}[t]
\centering
\begin{minipage}[b]{0.3\textwidth}
\centering
 \scalebox{0.75}{\begin{tikzpicture}
    % Nodes
    \node[circle, draw, fill=black, inner sep=1.5pt, label=left:$s_1$] (s1) at (0,1.5) {};
    \node[circle, draw, fill=black, inner sep=1.5pt, label=above:$u$] (u) at (2,1.5) {};
    \node[circle, draw, fill=black, inner sep=1.5pt, label=right:$t_1$] (t1) at (4,1.5) {};
    \node[circle, draw, fill=black, inner sep=1.5pt, label=left:$s_2$] (s2) at (0,0) {};
    \node[circle, draw, fill=black, inner sep=1.5pt, label=above:$v$] (v) at (2,0) {};
    \node[circle, draw, fill=black, inner sep=1.5pt, label=right:$t_2$] (t2) at (4,0) {};
    \node[circle, draw, fill=black, inner sep=1.5pt, label=left:$s_3$] (s3) at (0,-1.5) {};
    \node[circle, draw, fill=black, inner sep=1.5pt, label=above:$w$] (w) at (2,-1.5) {};
    \node[circle, draw, fill=black, inner sep=1.5pt, label=right:$t_3$] (t3) at (4,-1.5) {};
    
    % Edges
    \draw[->] (s1) -- (u);
    \draw[->] (s1) -- (v);
    \draw[->] (s2) -- (u);
    \draw[->] (s2) -- (w);
    \draw[->] (s3) -- (v);
    \draw[->] (u)--(t2);
    \draw[->] (v)--(t3);
    \draw[->] (v)--(t1);
    \draw[->] (w) -- (t2);
\end{tikzpicture}}
\vspace{0em} 

\footnotesize{(a) Layered digraph $D$}
\end{minipage}
\begin{minipage}[b]{0.6\textwidth}
\centering
 \scalebox{0.65}{
\begin{tikzpicture}[yscale=-1.6,xscale=1.5,thick,
			every transition/.style={draw=red,fill=red!20,minimum size=4.8mm},
			every place/.style={draw=blue,fill=blue!20, minimum size=6mm}, 
   			trans/.style={opacity=0.0, minimum size=6mm},
            dot/.style={inner sep=0pt, fill=black,circle,minimum size=3pt},
			 every fit/.style={draw,inner sep=4pt}]
			\tikzset{every loop/.style={min distance=6mm,in=-60,out=60,looseness=6}}
			every transition2/.style={draw=yellow,fill=yellow!20,minimum size=4.8mm},
             \node[trans](s01) at (-4.8,-0.5){};
             \node[trans](s02) at (-4.8,2){};
			\node[place] (s0) at (-4.8,0.7) {$s_0$};
			%\node [circle,draw](IRB) at (-1,1) {$ ~~~I_{BR}~~~ $};
			%\node [rectangle,draw,label=above:{$r=1$}](fix) at (-1,-1) {fixed MDP};
			\node[place] (s1) at (-2.2,-0.9) {$s_1$};
                \node[place] (s2) at (-2.2,0.7) {$s_2$};
			\node[place] (s3) at (-2.2,2.3) {$s_{3}$};
			\node[place, label=right:${~\frac{\delta}{\varepsilon}/\textcolor{red}0}$] (t1) at (2,-0.9) {$t_1$};
			\node[place,  label=right:${~\frac{\delta}{\varepsilon^2}/ \textcolor{red}0}$] (t2) at (2,0.7) {$t_2$};
                \node[place,label=right:${~\frac{\delta}{\varepsilon^3}/\textcolor{red}0}$] (t3) at (2,2.3) {$t_3$};
			\node[transition] at (-3.8,0.7) {$ a_0 $} edge [pre] (s0) edge [post] node[fill=white]{$ \frac{\varepsilon }{\delta}$} (s1) edge [post]  node[fill=white]{$ \frac{{\varepsilon}^2}{\delta} $} (s2) edge [post]  node[fill=white]{$\frac{\varepsilon^3}{\delta}$} (s3);
            \node[place] (v1) at (-0.2,-0.9) {$u$};
            \node[place] (v2) at (-0.2,0.7) {$v$};
            \node[place] (v3) at (-0.2,2.3) {$w$};
            \node[transition, scale=0.6] (a11) at (-1.2,-0.9) {}  edge [pre] (s1) edge [post] (v1);
            \node[transition, scale=0.6] (a12) at (-1.2,-0.4) {}  edge [pre] (s1) edge [post] (v2);
            \node[transition, scale=0.6] (a21) at (-1.2,0.2) {}  edge [pre] (s2) edge [post] (v1);
            \node[transition, scale=0.6] (a22) at (-1.2,1.2) {}  edge [pre] (s2) edge [post] (v3);
            \node[transition, scale=0.6] (a31) at (-1.2,1.8) {}  edge [pre] (s3) edge [post] (v2);
            \node[transition, scale=0.6] (a41) at (0.8,-0.9) {}  edge [pre] (v1) edge [post] (t2);
            \node[transition, scale=0.6] (a42) at (0.8,1) {}  edge [pre] (v2) edge [post] (t3);
            \node[transition, scale=0.6] (a51) at (0.8,0.2) {}  edge [pre] (v2) edge [post] (t1);
            \node[transition, scale=0.6] (a61) at (0.8,1.8) {}  edge [pre] (v3) edge [post] (t2);
\end{tikzpicture}}
\vspace{0em} 

\footnotesize{(b) Constructed instance of {\rewardP}}
\end{minipage}
\smallskip
\caption{The construction for the reduction from \textsc{Vertex-Disjoint Paths} to {\rewardP} in the proof of \cref{thm:reward-inapproximability}.}
\label{k-DVP}
\end{figure}

Next, we show that there exists a policy $\pi$ achieving worst-case reward $\hat{R}(\pi) = 1$ if the \textsc{Vertex-Disjoint Path} instance is a `yes' instance and that no policy can achieve worst-case reward larger than $\varepsilon < \alpha$ if the \textsc{Vertex-Disjoint Path} instance is a `no' instance. Hence, an $\alpha$-approximation for {\rewardP} would be able to solve \textsc{Vertex-Disjoint Path} in polynomial time, implying $P = N\!P$.

Consider any policy $\pi$ and observe that $\pi$ prescribes for each vertex $v \in V$ an outgoing arc $e \in \delta^{+}(v)$ with $\pi(v) = a_e$. 
Hence, for each $i \in L$, there is a unique path $P_i$ starting at $s_i$, following the arcs prescribed by $\pi$ and ending in $t_{j_i}$ for some $j_i \in L$.
Note that if the process reaches state $s_i$ for $i \in L$ under policy $\pi$, then it deterministically follows the prescribed path $P_i$ ending in state $t_{j_i}$.
From this, we can see that the probability that under policy $\pi$ the process terminates in state $t_j$ for $j \in L$ is $p_{j} := \sum_{i \in L : i_j = j} p(s_i|s_0, a_0)$. 
As there are $\ell$ terminal states and the rewards of $k = \ell - 1$ of them can drop to $0$ simultaneously, the worst-case reward of policy $\pi$ is given by
$\hat{R}(\pi)=\min\big\{p_1 \frac{\delta}{\varepsilon^1}, \dots, p_{\ell} \frac{\delta}{{\varepsilon}^{\ell}}\big\}$.

Now assume that $\hat{R}(\pi) \geq \alpha$.
This implies that $p_j \geq \frac{\varepsilon^j}{\delta} \alpha > \frac{\varepsilon^{j+1}}{\delta}$ for all $j \in L$.
In particular, for each $j \in L$ there must be at least one $i \in L$ with $j_i = j$. 
Thus, by counting, for each $j \in L$ there is a unique $i_j \in L$ with $j_{i_j} = j$.
But then $\frac{\varepsilon^{j+1}}{\delta} < p_j = \frac{\varepsilon^{i_j}}{\delta}$, which implies that $i_j \leq j$.
From this we conclude that $i_j = j$ for all $j \in L$, or in other words, each path $P_i$ for $i \in L$ is an $s_i$-$t_i$-path.
Moreover, since $j_i \neq j_{i'}$ for $i \neq i'$, the paths $P_i$ and $P_{i'}$ must be vertex-disjoint (if they intersected at a common vertex, they would follow the same arcs from that vertex onward and end at the same $t_j$).
So if there is a policy $\pi$ with $\hat{R}(\pi) \geq \alpha$ then the \textsc{Vertex-Disjoint Path} instance is a `yes' instance.

Conversely, the above arguments imply that if there are vertex disjoint $s_i$-$t_i$-paths $P_1, \dots, P_{\ell}$, then the policy defined by $\pi(v) = a_e$, where $e$ is the unique outgoing arc of $v$ on $P_i$ if $v \in V(P_i)$ or an arbitrary outgoing arc otherwise, yields a worst-case reward of $\hat{R}(\pi) = \min \{p_1 \frac{\delta}{\varepsilon^1}, \dots, p_{\ell}\frac{\delta}{\varepsilon^\ell}\} = 1$ because $j_i = i$ for all $i \in L$ for this policy.
\end{proof}
    
\begin{remark}\label{rem:history-gap}
The construction in the proof of \cref{thm:reward-inapproximability} also shows that the gap between the optimal reward obtainable by a history-independent deterministic policy and that obtainable by a history-dependent deterministic policy cannot be bounded by any constant.
Indeed, consider the instance of {\rewardP} constructed in the proof, assuming that the underlying digraph $D$ contains an $s_i$-$t_i$-path $P_i$ for every $i \in L$, but that it is not possible to select these paths so that they are pairwise vertex-disjoint (this is, e.g., true in the graph depicted in \cref{k-DVP}(a)).
Then, by the argument given in the proof of the theorem, no history-independent deterministic policy can achieve a worst-case reward larger than $\varepsilon$.
However, in the same instance, a history-dependent deterministic policy can obtain a worst-case reward of $1$ as follows: At the initial state $s_0$, choose action $a_0$; this leads to some state $s_i$ for some $i \in L$. Depending on this $i$, choose at any intermediate state $v$ corresponding to a node $v$ on $P$ the action corresponding to the unique out-arc of $v$ in $P$ (states not corresponding to nodes on $P$ are never reached and thus the chosen action does not matter).
It is easy to see that the resulting history-dependent policy leads to final state $t_i$ if and only if the state following the initial state $s_0$ is $s_i$, which happens with probability $\frac{\varepsilon^i}{\delta}$. Hence under the nominal rewards, the expected profit from terminating at $t_i$ is $1$ for every $i$. As $k < \ell$, there is at least one reward for which the realized reward equals the nominal reward, and hence the described policy obtains worst-case reward $1$, exceeding the reward of the history-independent policy by a factor of $\frac{1}{\varepsilon}$.
\end{remark}

We close this section by establishing the connection between {\rewardP} with two stages (but arbitrary uncertainty budget) and the \textsc{Santa Claus} problem.

\thmSanta*

\begin{proof}
    An instance of the \textsc{Santa Claus} problem consists of a set of jobs $J$ and a set of machines $M$, together with a processing time $p_{ij}$ for each $i \in M$ and $j \in J$.
    The goal is to find an assignment $\sigma: J \rightarrow M$ of each job to a machine so as to maximize the minimum load of any machine, i.e., $\max_{\sigma} \min_{i \in M} \sum_{j : \sigma(j) = i} p_{ij}$.
    Without loss of generality, we can assume that $p_{ij} \leq 1$ for each $i \in M$ and $j \in J$.

    Given an instance of \textsc{Santa Claus}, we construct an instance of {\rewardP} with $T=2$ such that there is a one-to-one correspondence between policies for the MDP corresponds and assignments of the jobs, with the robust reward of the policy being equal to the minimum load of any machine in the assignment.
    
    The MDP has an initial state $s_0$, an intermediate state $s_j$ for every $j \in J$, and a terminal state $s_i$ for each $i \in M$ as well as an additional terminal state $t_0$.
    The initial state has only a single action $a$ such that $p(s_j | s_0, a) = \frac{1}{|J|}$.
    Each intermediate state $s_j$ for $j \in J$ has an action $a_i$ for each $i \in M$ with $p(s_i | s_j, a_i) = p_{ij}$, $p(t_0 | s_j, a_i) = 1 - p_{ij}$, and $p(s_{i'} | s_j, a_i) = 0$ for all $i' \in M \setminus \{i\}$.
    Each terminal state $s_i$ for $i \in M$ has nominal reward $r(s_i) = |J|$ and alternative reward $r'(s_i) = 0$. Moreover, $r(t_0) = r'(t_0) = 0$.
    The uncertainty budget is $k = |M| - 1$.

    Note that every assignment $\sigma$ induces a policy $\pi$ via $\pi(s_j) = s_{\sigma(j)}$ and vice versa.
    Moreover, for each $i \in M$ the probability that this policy reaches terminal state $s_i$ equals $\sum_{j \in J : \sigma(j) = i} \frac{1}{|J|} p_{ij}$ and thus the expected nominal reward obtained at terminal state $i$ equals the load of machine $i$ in $\sigma$.
    The worst case realization of the terminal rewards sets the rewards of the $|M|-1$ states $s_i$ with the highest expected nominal reward under $\pi$ to $0$ and thus $\hat{R}(\pi) = \min_{i \in M} \sum_{j \in J : \sigma(j) = i} p_{ij}$, i.e., the expected robust reward of the policy equals the minimum load of a machine in the assignment.
\end{proof}

\section{Approximation for \textsc{2-Stage Robust MDP}}
\label{sec:approximation}

In this section, we present an approximation algorithm for \textsc{2-Stage Robust MDP}, proving \cref{thm:approximation}:

\thmApproximation*

We first provide an overview of the algorithm in \cref{sec:approximation-overview} and state worst-case guarantees for two subroutines of the algorithm based on two different relaxations, which together yield the desired approximation factor of our algorithm.
We then describe the two relaxations for the problem and the corresponding subroutines of the algorithm in detail in \cref{sec:approximation-knapsack,sec:approximation-gap}, respectively.
\subsection{Overview of the algorithm}
\label{sec:approximation-overview}

Our approximation algorithm for {\twostageP} is based on two subroutines, which in turn are based on two relaxations of the problem, one using a \textsc{Knapsack Cover} instance and one using a \textsc{Generalized Assignment} instance, respectively.
To analyze these two subroutines, we define the \emph{loss} of a policy $\pi$ as 
\begin{align*}
    L(\pi) := R(\pi) - \hat{R}(\pi),
\end{align*}
i.e., the drop in reward achieved by the policy in a worst-case scenario compared to its nominal reward.
As we will see, the first subroutine provides a good approximation when the loss in an optimal policy is a significant fraction of its nominal reward, while the other provides a good approximation when the loss is much smaller than the nominal reward.
We formalize this insight in the following two theorems, which we will prove in \cref{sec:approximation-knapsack,sec:approximation-gap}, respectively.

\begin{restatable}{theorem}{thmAlgoKnapsack}\label{thm:algo-knapsack}
	For any $\varepsilon > 0$, there exists an algorithm that, given an instance of {\twostageP}, computes in polynomial time a policy $\pi_1$ with $\hat{R}(\pi_1) \geq  \frac{1}{1+\varepsilon} \min \{\hat{R}(\pi), L(\pi)\}$ for all policies $\pi$.
\end{restatable}

\begin{restatable}{theorem}{thmAlgoGap}\label{thm:algo-gap}
	For any $\varepsilon > 0$, there exists an algorithm that, given an instance of {\twostageP},  computes in polynomial time a policy $\pi_2$ with $\hat{R}(\pi_2) \geq  \frac{1}{2}R(\pi) - 2\cdot(1+\varepsilon) L(\pi)$ for all policies $\pi$.
\end{restatable}

Our approximation algorithm simply runs both the algorithm from \cref{thm:algo-knapsack} and from \cref{thm:algo-gap} and returns whichever of the two policies $\pi_1$ and $\pi_2$ exhibits a higher worst-case reward.
\cref{thm:approximation} then directly follows by the following lemma, showing that at least one of the policies obtains approximately a fifth of the optimal worst-case reward.

\begin{lemma}
	Let $\varepsilon > 0$ and set $\varepsilon_1 := \frac{\varepsilon}{5}$ and $\varepsilon_2 := \frac{\varepsilon}{10 + 2\varepsilon}$.
	Let $\pi_1$ and $\pi_2$ be the policies constructed by the algorithms in \cref{thm:algo-knapsack,thm:algo-gap}, respectively, using $\varepsilon_1$ and $\varepsilon_2$, respectively, for a given instance of {\twostageP}.
	Then $\max \{\hat{R}(\pi_1), \hat{R}(\pi_2)\} \geq \frac{1}{5 + \varepsilon} \hat{R}(\pi)$ for all policies $\pi$.
\end{lemma}

\begin{proof}
	Let $\pi$ be any policy for the given instance of {\twostageP}.
	We distinguish two cases.
	First, if $L(\pi) \geq \frac{1}{5} \hat{R}(\pi)$, then \cref{thm:algo-knapsack} yields 
	\begin{align*}
		\hat{R}(\pi_1) \geq \frac{1}{1+\varepsilon_1} \min \{\hat{R}(\pi), L(\pi)\} 
		\geq \frac{1}{1+\varepsilon_1} \cdot \frac{1}{5} \cdot \hat{R}(\pi) = \frac{1}{5 + \varepsilon} \cdot \hat{R}(\pi).
	\end{align*}
	Second, if $L(\pi) < \frac{1}{5} \hat{R}(\pi)$, then 
	\begin{align*}
		\hat{R}(\pi_2) 
		& \; \geq \; \frac{1}{2} R(\pi) - 2(1 + \varepsilon_2) L(\pi)
		\; = \; \frac{1}{2} (\hat{R}(\pi) + L(\pi)) - 2(1 + \varepsilon_2) L(\pi) 
		\\
	  & \; = \; \frac{1}{2} \hat{R}(\pi)  - \left(\frac{3}{2} + 2\varepsilon_2\right) L(\pi) 
	  \; > \; \frac{1 - 2 \varepsilon_2}{5}\hat{R}(\pi)
	  \; = \; \frac{1}{5 + \varepsilon}\hat{R}(\pi)
	 \end{align*} where the first inequality follows from \cref{thm:algo-gap} and the first identity follows from $\hat{R}(\pi) = R(\pi) - L(\pi)$.
\end{proof}

\paragraph{Assumptions simplifying the presentation}
Before we describe our subroutines, we introduce some assumptions that are without loss of generality but help to simplify the presentation of the subroutines.
\begin{enumerate}
    \item We assume that $A(s_0) = \{a_0\}$, i.e., there is only a single action $s_0$ available at the initial state.
    This is without loss of generality as we can simply run the algorithm separately for each fixed choice of action at $s_0$ and then return the best policy computed among these separate runs.
    \item We assume that $r'(t) = 0$ for all $t \in S_2$, i.e., the alternative reward at any terminal state is $0$. 
    This is without loss of generality as we can apply the following modification for any terminal state $t \in S_2$ with $r'(t) > 0$: 
    \begin{itemize}
        \item Let $h := \left\lceil\frac{r'(t)}{r(t) - r'(t)}\right\rceil$ and replace $t$ by $h+1$ states $t_0, \dots, t_h$ with rewards $r(t_0) = (r(t) - r'(t))(h+1)$ and $r'(t_0) = 0$, as well as $r(t_i) = \frac{h+1}{h} r'(t)$ and $r'(t_i) = 0$ for $i \in \{1, \dots, h\}$.
        \item Set $p(t_i|s,a) = \frac{1}{h+1} p(t|s,a)$ for all $s \in S_1$ and $a \in A(s)$.
    \end{itemize} 
    Note that this construction preserves the value of the nominal reward of any policy as $\sum_{i=0}^{h} p(t_i|s,a) r(t_i) = p(t|s,a) r(t)$ for any $s \in S_1$ and $a \in A(s)$.
    It also preserves the value of the worst-case reward for any policy, as $\sum_{i=1}^{h} p(t_i|s,a) r(t_i) = p(t|s,a) r'(t)$ for any $s \in S_1$ and $a \in A(s)$, i.e., the deviation of the reward at $t_0$ in the modified instance has the same effect as the deviation of the reward at $t$ in the original instance, and no deviation at $t_i$ for $i \in \{1, \dots, h\}$ in the modified instance has a worse effect on the reward, because $p(t_0|s,a)r(t_0) \geq p(t_i|s,a)r(t_i)$ for all $i \in \{1, \dots, h\}$, $s \in S_1$ and $a \in A(s)$. 
\end{enumerate}

\subsection{Algorithm~1: \textsc{Knapsack Cover}}
\label{sec:approximation-knapsack}

Our first algorithm is based on a \textsc{Knapsack Cover} relaxation of {\twostageP}.
The underlying idea is as follows: Assume we had access to a value $L$ which is close to the loss $L(\pi^*)$ of an optimal policy $\pi^*$, as well as to the terminal state $\hat{t}$ affected by the worst-case uncertainty realization under $\pi^*$ (indeed, we will show below that such values can be obtained by enumerating over a set of polynomial size).
Using~$L$ and $\hat{t}$ as parameters, we formulate an integer program (IP) whose goal is to find a policy maximizing the nominal reward obtained at states other than $\hat{t}$ while ensuring that the nominal reward at $\hat{t}$ is at least~$L$. 
Using the FPTAS for \textsc{Knapsack Cover}, we obtain a near optimal solution, of which we show that it induces a policy $\pi_1$ with the property that $\hat{R}(\pi_1) \geq \frac{1}{1 + \varepsilon} \min \{\hat{R}(\pi^*), L(\pi^*)\}$, where $\varepsilon$ is the precision of the FPTAS for \textsc{Knapsack Cover} and our ``guess'' for $L$.

To introduce the IP, we define 
\begin{align*}
    v^a_{st} := p(s | s_0, a_0) \cdot p(t|s,a)\cdot r(t)
\end{align*}
for $s \in S_1$, $t \in S_2$, and $a \in A(s)$.
For any $L > 0$ and $\hat{t} \in S_2$, consider the following IP:
\begin{align*}
	\mathrm{UB_1}(L, \hat{t}) \qquad &\mathrm{max} \quad \textstyle \sum_{s \in S_1} \sum_{a \in A(s)}\sum_{t \in S_2 \setminus \{\hat{t}\}} v^{a}_{st} x_{sa} \\
	&\mathrm{s.t.} \quad 
	\begin{array}[t]{r l l l l r}
		\sum_{s \in S_1} \sum_{a\in A(s)} v^{a}_{s\hat{t}} \, x_{sa} & \geq L &{}&  &{}&\\
		\sum_{a \in A(s)} x_{sa}&=1 &{}&   \forall\, s \in S_1 &{}&\\
		x_{sa}&\in \{0,1\}  &{}& \forall\, s \in S_1, ~ a \in A(s)&{}&
	\end{array}
\end{align*}

We first observe that $\mathrm{UB_1}$ with the right parameters $L$ and $\hat{t}$ indeed yields an upper bound on the worst-case reward of a policy.

\begin{lemma}\label{lem:ub-knapsack}
	Let $\pi$ be a policy. For any $L \leq L(\pi)$ there exists a $\hat{t} \in S_2$ such that the optimal value of $\mathrm{UB_1}(L, \hat{t})$ is at least $\hat{R}(\pi)$.
\end{lemma}
\begin{proof}
	 Let $\hat{r} \in \hat{\mathcal{Q}}$ be a worst-case scenario for $\pi$ and let $\hat{t} \in S_{2}$ be the state with $\hat{r}(\hat{t}) = 0 < r(\hat{t})$ (recall that $r'(t) = 0$ for all $t \in S_2$ by the assumptions made on the instance).
	 We define a solution $x$ to $\mathrm{UB_1}(L, \hat{t})$ by setting
	 $x_{s\pi(s)}= 1$ and $x_{sa} = 0$ for all $s \in S_1$ and $a \in A(s) \setminus \{\pi(s)\}$. 
	 Note that $x$ is feasible because 
     \begin{align*}
        \sum_{s \in S_1}\sum_{a\in A(s)} v^{a}_{s\hat{t}}\, x_{sa} = \sum_{s \in S_1} p(s | s_0, a_0) \cdot p(\hat{t} |s,\pi(s)) \cdot r(\hat{t}) = L(\pi) \geq L.
     \end{align*}
	 Moreover, 
	 \begin{align*}
	 		\sum_{s \in S_1} \sum_{a \in A(s)}\sum_{t \in S_2 \setminus \{\hat{t}\}} v^{a}_{st} x_{sa} = \sum_{t \in S_2 \setminus \{\hat{t}\}} \sum_{s \in S_1} p(s | s_0, a_0) \cdot p(t|s,\pi(s)) \cdot r(t) = \hat{R}(\pi),
	 \end{align*}
	completing the proof of the lemma.
\end{proof}

Conversely, every solution $x$ to $\mathrm{UB_1}(L, \hat{t})$ also yields a policy $\pi_x$ by defining $\pi_x(s) = a$ for the unique action $a \in A(s)$ with $x_{sa} = 1$.
The following lemma gives a guarantee on the worst-case reward achieved by~$\pi_x$.

\begin{lemma}\label{lem:lb-knapsack}
	Let $L > 0$ and $\hat{t} \in S_2$, and let $x$ be a feasible solution to $\mathrm{UB_1}(L, \hat{t})$ with objective function value~$V$. Then $\hat{R}(\pi_x) \geq \min \{V, L\}$.
\end{lemma}

\begin{proof}
	For $t \in S_2$ let 
    \begin{align*}
        Q_x(t) := \sum_{s \in S_1} p(s | s_0, a_0) \cdot p(t|s,\pi_x(s)) \cdot r(t) = \sum_{s \in S_1} \sum_{a \in A(s)} v^a_{st} x_{sa}
    \end{align*}
    be the nominal reward obtained at terminal state $t$ by policy $\pi_x$ in expectation.
	Note that for any $t' \in S_2$, the reward obtained by $\pi_x$ under the scenario~$\hat{r}$ in which $\hat{r}(t') = 0$ is  
	\begin{align*}
		\sum_{t \in S_2 \setminus \{t'\}} Q_x(t) & = \sum_{s \in S_1} \sum_{a \in A(s)} \sum_{t \in S_2 \setminus \{t'\}} v^{a}_{st} x_{sa},
	\end{align*}
	which is equal to $V$ if $t' = \hat{t}$, and at least $Q_x(t') = \sum_{s \in S_1} \sum_{a \in A(s)} v^a_{s\hat{t}}\,x_{sa} \geq L$ otherwise. 
\end{proof}

Using \cref{lem:ub-knapsack,lem:lb-knapsack}, we are now ready to describe the algorithm that proves the guarantee stated earlier in \cref{thm:algo-knapsack}.

\thmAlgoKnapsack*

\begin{proof}
For $L > 0$ and $\hat{t} \in S_2$, let $V(L, \hat{t})$ denote the value of an optimal solution to $\mathrm{UB_1}(L, \hat{t})$.
Note that $\mathrm{UB_1}(L, \hat{t})$ is a \textsc{Knapsack Cover} problem and thus admits an FPTAS~\citep{pruhs2007approximation}.
In particular,  we can compute in polynomial time a feasible solution $x_{L,\hat{t}}$ to $\mathrm{UB_1}(L, \hat{t})$ whose objective value is at least $\frac{1}{1+\varepsilon}V(L, \hat{t})$.
Let $\pi_{L,\hat{t}} := \pi_{x_{L,\hat{t}}}$ be the corresponding policy.

To obtain the policy $\pi_1$ stated in the theorem, we run the following algorithm.
Let $r_{\min} := \min_{t \in S_2} r(t)$ and $r_{\max} := \max_{t \in S_2} r(t)$. Let $\ell := \lceil \log_{1+\varepsilon} \frac{|S_2| r_{\max}}{r_{\min}} \rceil$.
Let $L_i := (1+\varepsilon)^i \frac{r_{\min}}{|S_2|}$ for $i \in \mathbb{N}$.
For every $i \in \{0, \dots, \ell\}$ and every $\hat{t} \in S_2$, compute the policy $\pi_{L_i, \hat{t}}$ as described above.
Let $\pi_1 \in \argmax_{i \in \{0, \dots, \ell\}, \hat{t} \in S_{2}} \hat{R}(\pi_{L_i, \hat{t}})$, i.e., we return a policy $\pi_1$ with highest nominal reward among the ones computed.

Note that this algorithm runs in polynomial time, as there are only \mbox{$(\ell + 1)|S_2|$} combinations of $L$ and $\hat{t}$ to be checked, $\ell$ is polynomial in $\log |S_2|$ and the encoding size of $r_{\max}$ and $r_{\min}$, and for each fixed combination of $L$ and $\hat{t}$ the corresponding policy $\pi_{L_i, \hat{t}}$ can be computed in polynomial time.

Now consider an arbitrary policy $\pi$.
Note that $\frac{r_{\min}}{|S_2|} \leq L(\pi) \leq r_{\max}$ and hence there exists an $i' \in \{0, \dots, \ell\}$ such that $\frac{1}{1 + \varepsilon} L(\pi) \leq L_{i'} \leq L(\pi)$.
Hence, $V(L_{i'}, \hat{t}') \geq \hat{R}(\pi)$ by \cref{lem:ub-knapsack} for some $\hat{t}' \in S_2$.
We obtain
\begin{align*}
	\hat{R}(\pi_1) \geq \hat{R}(\pi_{L_{i'},\hat{t}'}) \geq \min \left\{ \frac{V(L_{i'},\hat{t}')}{1+\varepsilon}, L_{i'} \right\} \geq \frac{1}{1+\varepsilon} \min \{\hat{R}(\pi), L(\pi)\},
\end{align*}
where the first inequality follows from the choice of $\pi_1$, the second inequality follows from \cref{lem:lb-knapsack}
and the construction of $\pi_{L_{i'},\hat{t}'}$ from a near-optimal solution to $\mathrm{UB_1}(L_{i'}, \hat{t}')$, and the third inequality follows from 
$V(L_{i'}, \hat{t}') \geq \hat{R}(\pi)$
and
$L_{i'} \geq \frac{1}{1 + \varepsilon} L(\pi)$.
\end{proof}

\subsection{Algorithm~2: \textsc{Generalized Assignment}}
\label{sec:approximation-gap}

Our second algorithm is based on the \textsc{Generalized Assignment} problem.
While the algorithm described in the preceding section yields a good approximation factor when $L(\pi^*)$ is close to $\hat{R}(\pi^*)$ for some optimal policy, its guarantee deteriorates when $L(\pi^*)$ is significantly smaller.
The second algorithm is thus aimed at the case where $L(\pi^*)$ is small compared to $\hat{R}(\pi^*)$.

Recalling the definition of $v^a_{st} := p(s | s_0, a_0) \cdot p(t|s,a)\cdot r(t)$ from the previous section, we define
\begin{align*}
    \textstyle v_{st} & := \max_{a \in A(s)} v^a_{st} ~\text{ and }\\
    v_{st^*} & \textstyle := \max_{a \in A(s)} \sum_{t \in S_2 \setminus \{t_s^a\}} v^a_{st} ~ \text{ with } ~ t_s^a \in \argmax_{t \in S_2} v^a_{st},
\end{align*}
for $s \in S_1$ and $t \in S_2$, 
where $t^*$ is an artificial state used in the relaxation.
Given $L > 0$, consider the following linear program (LP):
\begin{align*}
	\mathrm{UB_2}(L) \qquad & \mathrm{max} \quad \sum_{s \in S_1} \sum_{t\in S_2 \cup \{t^*\}} v_{st}y_{st} \\ 
    & \mathrm{s.t.} \quad 
	\begin{array}[t]{rllllr}
		\sum\limits_{s \in S_1} v_{st}y_{st}&\leq L &{}&   \forall\; t \in S_2  &{}&  \\
		\sum_{t \in  S_2 \cup \{t^*\}} y_{st}&=1 &{}&  \forall\; s \in S_1 &{}&  \\
		y_{st}&= 0 &{}& \forall s \in S_1, t \in S_2 \text{ with } v_{st}>L  &{}&  \\
		y_{st}&\geq 0  &{}& \forall~ s \in S_1, t\in S_2\cup \{t^*\}  &{}& 
	\end{array}
\end{align*}

Intuitively, $\mathrm{UB_2}(L)$ relaxes {\rewardP} by considering for each $s \in S_1$ the following possibilities: Either we only consider the reward $v_{st}^a$ obtained at a single terminal state $t$ under some action $a$ taken at $s$, or we consider the expected reward $v_{st}^a$ obtained at all terminal states $t$ under some action $a$, but discarding the largest among these expected rewards.
While in the latter case, the relaxation already incorporates the possible loss due to deviating rewards, in the former case, we further bound the concentration of the expected reward at any terminal state by $L$.
The following lemma shows that $\mathrm{UB_2}(L)$ indeed yields an upper bound on the nominal reward of policies with loss at most $L$.

\begin{lemma}\label{lem:ub-gap}
For any policy $\pi$ and any $L \geq L(\pi)$, the optimal value of $\mathrm{UB_2}(L)$ is at least $\frac{1}{2}R(\pi)$.
\end{lemma}
\begin{proof}
Let $\pi$ be a policy. 
We define a fractional solution $y$ to $\mathrm{UB_2}(L)$ as follows.
For each $s \in S_1$, let $t_s := t^{\pi(s)}_s \in \argmax_{t \in S_2} v^{\pi(s)}_{st}$.
If $v_{st_s} \geq v_{st^*}$, then set 
$y_{st_s} := \frac{v^{\pi(s)}_{st_s}}{v_{st_s}}$, 
$y_{st^*} := 1-\frac{v^{\pi(s)}_{st_s}}{v_{st_s}}$, 
and $y_{st} := 0$ for $t \in S_2 \setminus \{t_s, t^*\}$.
Otherwise, $ y_{st^*}:=1$ and $y_{st}:=0$ for $t \neq t^*$. 	

Note that, by construction, $\sum_{t \in S_{2} \cup \{t^*\}} y_{st} = 1$ for all $s \in S_1$ and $y_{st} = 0$ for all $s \in S_1$, $t \in S_2$ with $v_{st} > 0$.
Because the largest nominal reward obtained in $\pi$ at any terminal state is no more than $L(\pi)$, we have
\begin{align*}
	\sum_{s\in S_1} v_{st}y_{st} \leq \sum_{s \in S_1: t = t_s} v^{\pi(s)}_{st} \leq \sum_{s \in S_1}  p(s | s_0, a_0) \cdot p(t|s, \pi(s))\cdot r(t) \leq L(\pi) \leq L
\end{align*}
for all $t \in S_2$.
Thus $y$ is a feasible solution to $\mathrm{UB_2}(L)$.

Note that by definition of $t_s = t^{\pi(s)}_s$ and $v_{st^*} = \max_{a \in A(s)} \sum_{t \in S_2 \setminus \{t^{a}_s\}} v^a_{st}$, we obtain
\begin{align*}
    v^{\pi(s)}_{st_s} + v_{st^*} & \;\geq\; v^{\pi(s)}_{st_s} + \sum_{t \in S_2 \setminus \{t^{\pi(s)}_s\}} v^{\pi(s)}_{st} = \sum_{t \in S_2} v^{\pi(s)}_{st} \\
    & \;=\;  \sum_{t \in S_2} p(s | s_0, a_0) \cdot p(t|s, \pi(s))\cdot r(t)
\end{align*}
for all $s \in S_1$, from which we conclude that
\begin{align*} 
	\sum_{s \in S_1} \sum_{t \in S_2 \cup \{t^*\}} v_{st}y_{st} & \;\geq\; \sum_{s \in S_1} \max \{v^{\pi(s)}_{st_s}, v_{st^*}\}
 \;\geq\; \frac{1}{2} \sum_{s \in S_1} (v^{\pi(s)}_{st_s} + v_{st^*}) \\
 & \;\geq\; \frac{1}{2} \sum_{s \in S_1} \sum_{t \in S_2}  p(s | s_0, a_0) \cdot p(t|s, \pi(s))\cdot r(t) \;=\; \frac{1}{2}R(\pi),
\end{align*}
completing the proof of the lemma.
\end{proof}	

Conversely, any integer solution $y$ to $\mathrm{UB_2}(L)$ also induces a policy $\pi_y$ via the following construction:
For $s \in S_1$, let $t^y_s \in S_2 \cup \{t^*\}$ be the unique (possibly artificial) terminal state with $y_{st^y_s} = 1$.
Then set 
\begin{align*}
    \pi_{y}(s) \in \argmax_{a \in A(s)} \sum_{t \in S_2 \setminus \{t^a_s\}} v^a_{st} ~\text{ if } t^y_s = t^*, 
\text{ and }
\pi_{y}(s) \in \argmax_{a \in A(s)} v^a_{st^y_s} ~\text{ otherwise.}
\end{align*}
We give the following guarantee on the worst-case reward achieved by $\pi_y$.

\begin{lemma}\label{lem:lb-gap}
	Let $L > 0$ and let $y$ be a feasible solution to 	$\mathrm{UB_2}(L)$ with objective function value $W$.
	Then $\hat{R}(\pi_y) \geq W - L$.
\end{lemma}

\begin{proof}
	Let $S_1' := \{s \in S_1 : t^y_s \neq t^*\}$ and $S_1'' := \{s \in S_1 : t^y_s = t^*\}$.
	Note that
	\begin{align*}
		R(\pi_y) & 
		\;=\; \sum_{s \in S_1} \sum_{t \in S_2} p(s | s_0, a_0) \cdot p(t|s,\pi_y(s)) \cdot r(t) 
		\;=\; \sum_{s \in S_1} \sum_{t \in S_2} v^{\pi(s)}_{st} \\
		& \;=\; \sum_{s \in S_1'} \sum_{t \in S_2} v^{\pi(s)}_{st} +
		\sum_{s \in S_1''} \sum_{t \in S_2} v^{\pi(s)}_{st}
		\;\geq\; \sum_{s \in S_1'} v^{\pi(s)}_{st^y_s} +
		\sum_{s \in S_1''} \sum_{t \in S_2} v^{\pi(s)}_{st}
		\\
		& \;=\; \sum_{s \in S_1} \sum_{t \in S_2} v_{st} y_{st}
		+ \sum_{s \in S_1''} \sum_{t \in S_2} v^{\pi(s)}_{st}
	\end{align*}
	where the final identity follows from $v_{st^y_s}^{\pi(s)} = v_{st^y_s}$ by choice of $\pi(s)$ when $t^y_s \neq t^*$.
	
	Let $\hat{t} \in S_2$.
	Then the reward obtained by $\pi_y$ in the scenario in which $\hat{r}_{\hat{t}} = 0$ is at least
	\begin{align*}
		\sum_{s \in S_1} \sum_{t \in S_2 \setminus \{\hat{t}\}} v_{st} y_{st}
		+ \sum_{s \in S_1''} \sum_{t \in S_2 \setminus \{\hat{t}\}} v^{\pi(s)}_{st} & \; \geq \; \left(\sum_{s \in S_1} \sum_{t \in S_2} v_{st} y_{st}\right) - L
		+ \sum_{s \in S_1''} v_{st^*}
		\\
		& \geq \; \left(\sum_{s \in S_1} \sum_{t \in S_2 \cup \{t^*\}} v_{st} y_{st}\right) - L = W - L,
	\end{align*}
	where the first inequality follows from $\sum_{s \in S_1} v_{s\hat{t}} \, y_{s\hat{t}} \leq L$ by feasibility of $y$ and from 
    \begin{align*}
        \sum_{t \in S_2 \setminus \{\hat{t}\}} v^{\pi(s)}_{st} \geq \sum_{t \in S_2} v^{\pi(s)}_{st} - \max_{t \in S_2} v^{\pi(s)}_{st} = v_{st^*}
    \end{align*} by choice of $\pi(s)$ for $s \in S_1''$ and definition of $v_{st^*}$.
	We conclude that $\hat{R}(\pi_y) \geq W - L$.
\end{proof}

We are now ready to prove \cref{thm:algo-gap}, providing the second subroutine for our algorithm.

\thmAlgoGap*

\begin{proof}
	First note that $\mathrm{UB_2}(L)$ for fixed $L > 0$ is identical to the LP relaxation of \textsc{Generalized Assignment} used in the algorithm of \citet{shmoys1993approximation} for that problem.
	Given a feasible (possibly fractional) solution $y$ to $\mathrm{UB_2}(L)$, this algorithm computes in polynomial time an integer solution $y'$ to $\mathrm{UB_2}(2L)$ with 
	\begin{align*}
		\sum_{s \in S_1} \sum_{t \in S_2 \cup \{t^*\}} v_{st} y'_{st} \geq \sum_{s \in S_1} \sum_{t \in S_2 \cup \{t^*\}} v_{st} y'_{st}.
	\end{align*}
	That is, $y'$ achieves the same objective function value as $y$, but requires doubling the value of $L$ to be feasible.
	For any $L > 0$, let $y_L$ denote an optimal (fractional) solution to $\mathrm{UB_2}(L)$, let $y'_L$ denote the corresponding integer solution computed by the algorithm of \citet{shmoys1993approximation}, and let $\pi_L := \pi_{y'_L}$ denote the policy obtained from applying \cref{lem:lb-gap} to $y'$.
	
	To obtain the policy $\pi_2$ stated in the lemma, we run the following algorithm. 
	Let $r_{\min} := \min_{t \in S_2} r_t$ and $r_{\max} := \max_{t \in S_2} r_t$. 
	Let $\ell := \lceil \log_{1+\varepsilon} \frac{|S_2| r_{\max}}{r_{\min}} \rceil$.
	Let $L_i := (1+\varepsilon)^i \frac{r_{\min}}{|S_2|}$ for $i \in \mathbb{N}$.
	For every $i \in \{0, \dots, \ell\}$ and every $\hat{t} \in S_2$, compute the policy $\pi_{L_i}$ as described above.
	Let $\pi_2 \in \argmax_{i \in \{0, \dots, \ell\}} \hat{R}(\pi_{L_i})$, i.e., we return a policy $\pi_2$ with highest nominal reward among the ones computed.
	
	Note that the algorithm runs in polynomial time, because there are only $\ell$ different values of $L_i$ for which the policy has to be computed,  $\ell$ is polynomial in $\log |S_2|$ and the encoding size of $r_{\max}$ and $r_{\min}$, and for each fixed $L_i$, the corresponding policy $\pi_{L_i}$ can be computed in polynomial time.

	Now consider an arbitrary policy $\pi$.
	Note that $\frac{r_{\min}}{|S_2|} \leq L(\pi) \leq r_{\max}$ and hence there exists an $i' \in \{0, \dots, \ell\}$ such that $(1 + \varepsilon) L(\pi) \geq L_{i'} \geq L(\pi)$.
Hence, $W(L_{i'}) \geq \frac{1}{2}R(\pi)$ by \cref{lem:ub-gap}, and we obtain
\begin{align*}
	\hat{R}(\pi_2) \geq \hat{R}(\pi_{L_{i'}}) \geq W(L_{i'}) - 2L_{i'} \geq \frac{1}{2}R(\pi) - 2 \cdot (1 + \varepsilon) L(\pi)
\end{align*}
where the first inequality follows from the choice of $\pi_2$, the second inequality follows from \cref{lem:lb-gap}
and the construction of $\pi_{L_{i'}}$ from a feasible integer solution to $\mathrm{UB_2}(2L_{i'})$ and the fact that 
$\sum_{s \in S_1} \sum_{t \in S_2 \cup \{t^*\}} v_{st} y'_{st} \geq W(L_{i'}, \hat{t}') \geq \hat{R}(\pi)$
and the last inequality follows from 
$L_{i'} \geq (1 + \varepsilon) L(\pi)$.
\end{proof}

\section{Complexity for LDST with uncertainty in the transition probabilities}
\label{sec:hardness}

We introduce the following notation to simplify the presentation of the two reductions in this section.

\paragraph{Notation}
In the reductions, we will only make use of actions $a \in A(s)$ for $s \in S$ of two types: 
\begin{itemize}
    \item actions with randomized outcomes but where the transition probabilities are known and not subject to uncertainty, i.e., \mbox{$\hat{\mathcal{P}}_{sa} = \{p(\cdot|s,a)\}$},
    \item actions with a deterministic but uncertain outcome, i.e., \mbox{$\hat{p}(t|s,a) \in \{0, 1\}$} for all $t \in S$.
\end{itemize}
To simplify the description of the latter type of actions, we introduce the notation $z(s, a) \in S$ and $Z(s,a) \subseteq S$ to describe an action for which $\hat{\mathcal{P}}_{sa}$ contains the nominal distribution given by $p(z(s,a)|s,a) = 1$ and alternate distributions $\hat{p}(z|s,a) = 1$ for each $z \in Z(s,a)$.

Our constructions will further contain state-action pairs $s \in S_{\tau}$ and $a \in A(s)$ with $p(t|s,a) > 0$ for some $t \in S_{\theta}$ with $\theta > \tau + 1$.
This is with the tacit understanding that such an action leads from $s$ to $t$ via a sequence of intermediate states $s_{\tau+1}, \dots, s_{\theta-1}$ that we omit in order to declutter the presentation.

%\subsection{$\Sigma_2^p$-hardness of {\transitionP}}
\label{sec:transition-sigma-2-p}

Using the notation defined above, we show that {\transitionP} is $\Sigma_2^p$-hard.

\thmTransitionSigma*

\begin{proof}
	We prove the theorem by a reduction from \textsc{Max-Min Vertex Cover}. In an instance of this problem, we are given a graph $G = (V, E)$ whose vertex set $V$ is partitioned into subsets $\{V_{i,j}\}_{i\in I,j\in J}$ for some finite sets $I$ and $J$.
	For a function $t: I \rightarrow J$, define $V_t := \bigcup_{i \in I} V_{i,t(i)} $ and let $ G_t = (V_t, E[V_t]) $ be the induced subgraph of $ G $ on the vertex set $V_t$.
	The decision version of \textsc{Max-Min Vertex Cover}, which is known to be $\Sigma_{2}^{p}$-complete \citep{ko1995complexity}, asks whether there exists a function $t : I \rightarrow J$ such that the size of a minimum vertex cover in $G_t$ is at least $\ell$ for a given value $\ell \in \mathbb{N}$.
	
	Before we begin our reduction, we establish the following assumptions that will simplify the description of our construction. We argue below that they are without loss of generality.
	\begin{enumerate} 
		\item There are no edges with one endpoint in $V_{i,j}$ and one endpoint in $V_{i,j'}$ for any $i \in I$ and any $j, j' \in J$ with $j \neq j'$.
		\item There exists an $\bar{m} \in \mathbb{N}$ such that $|E[V_{i, j}]| = \bar{m}$ and $|E[V_{i,j}:V_{i',j'}]| = \bar{m}$ for all $i, i' \in I$ with $i \neq i'$ and $j, j' \in J$.
	\end{enumerate}
	The first assumption is w.l.o.g.\ because the edges in question do not appear in $G_t$ for any function $t$ and we can thus delete them.
	To see that the second assumption is w.l.o.g., let $m_1 := \max_{i \in I, j \in J} |E[V_{i,j}]|$ be the maximum number of edges in the subgraph induced by any $V_{i,j}$, and let $m_2 := \max_{i, i' \in I, j, j' \in J, i \neq i'} |E[V_{i,j}:V_{i',j'}]|$ be the maximum number of edges between any two components $V_{i,j}$ and $V_{i',j'}$ with $i \neq i'$. Let $\bar{m} := \max \{m_1, m_2\}$. 
	For each $i \in I$ and $j \in J$, we introduce two dummy vertices $v'_{i,j},v''_{i,j}$ to each $V_{i, j}$.
	We further introduce $\bar{m} - |E[V_{i,j}]|$ parallel edges between $v'_{i,j}$ and $v''_{i,j}$ for each $i \in I$ and each $j \in J$, as well as $\bar{m} - |E[V_{i,j}:V_{i',j'}]|$ edges between $v'_{i,j}$ and $v'_{i',j'}$ for each $i, i' \in I$ with $i \neq i'$ and $j, j' \in J$.
	Note that the resulting graph fulfils the assumptions, and that for any $t : I \rightarrow J$, a vertex set $S$ in the original graph is a vertex cover in $G_t$ for the original instance if and only if $S \cup \{v'_{i,t(i)} : i \in I\}$ is a vertex cover for $G_t$ in the modified instance.
	
	We now proceed to construct an instance of {\transitionP} on a finite-horizon MDP from a given instance of \textsc{Max-Min Vertex Cover} fulfilling the assumptions mentioned above.
	Observe that by the second assumption above, the number of edges in $G_t$ equals 
    \begin{align*}
        m := |I|\bar{m} + \binom{|I|}{2}\bar{m} = \frac{|I|(|I| + 1)}{2}\bar{m}
    \end{align*} for any $t : I \rightarrow J$.
	
	We define the set of states as
	\begin{align*}
		S = \{s_0\} \cup \{s_i : i \in I\} \cup \{s_e : e \in E\} \cup \{s_v : v \in V\} \cup \{t_0, t_1\}.
	\end{align*}
	The set of actions at each state $s \in S \setminus \{t_0, t_1\}$ is given by
  \begin{alignat*}{2}
       A(s) & = \{a_j:\;j\in J\} & \quad \text{ if } s = s_i \text{ for some } i \in I,\\
       A(s) & = \{a_0\} & \text{ otherwise.}
  \end{alignat*}
  States $t_0$ and $t_1$ are terminal states with rewards $r(t_0)=0$ and $r(t_1)=1$. 

Let $\varepsilon = \frac{2\bar{m}}{2\bar{m}|I| + |V|}$.
The nominal transition probabilities for actions at states $s_0$ and $s_i$ for $i \in I$ are given as follows and not subjected to uncertainty:
\begin{alignat*}{3}
	 p(s_i|s_0,a_0) &= \varepsilon  && \quad \text{ for } i\in I,\\
	 p(s_v|s_0,a_0) & = \frac{1-\varepsilon|I|}{|V|} && \quad \text{ for } v \in V,\\
	 p(s_e|s_i,a_j) & = \frac{2}{3\bar{m}}  && \quad \text{ for } i \in I,\; j \in J, \text{ and } e \in E[V_{i,j}],\\ 
	 p(s_e|s_i,a_j) & = \frac{1}{3\bar{m}}  && \quad \text{ for } i \in I,\; j \in J, \text{ and } e \in \delta(V_{i,j}).
\end{alignat*}
Note that the probabilities at each state-action pair $(s_i, a_j)$ sum up to $1$ as $|E[V_{i,j}]| = |\delta(V_{i,j}| = \bar{m}$ by our earlier assumption.
 The actions at states $s_e$ for $e \in E$ and $s_v$ for $v \in V$ are uncertain deterministic with the following outcomes:
\begin{alignat*}{3}
	z(s_e, a_0) & = t_1, & \quad  Z(s_e, a_0) & = \{s_v, s_w\} && \quad \text{for } e = \{v, w\} \in E,\\
	z(s_v,a_0) & = t_1, & Z(s_v,a_0) & = \{t_0\} && \quad \text{for } v\in V.
\end{alignat*}
Finally, the uncertainty budget is given by $k \,=\, m + \ell - 1$.

	\begin{figure}[t] 
		\centering 
		\scalebox{0.85}{\begin{tikzpicture}[yscale=-1.6,xscale=1.5,thick,
			every transition/.style={draw=red,fill=red!20,minimum size=4.8mm},
			every place/.style={draw=blue,fill=blue!20, minimum size=6mm}]
			every transition2/.style={draw=yellow,fill=yellow!20,minimum size=4.8mm},
			\node[place] (s0) at (-4,0) {$s_0$};
			\node[place] (s1) at (-1.5,-1) {$s_1$};
			\node[place,label=above:{$e'\in \delta(V_{i,j})$}] (e1) at (0.5,-1) {$s_{e'}$};
			\node[place] (v1) at (2.5,0) {$s_{e}^1$};
			\node[place,label=right:{$1$}] (t1) at (4.5,0) {$t_1$};
			\node[place,label=right:{$0$ }] (t2) at (4.5,1.5) {$t_2$};
			%s1\label=above:{$t=1$}];
			% insert m actions for each state
			% t=1 
			\foreach \i in {2} {
				\node[place] (s\i) at (-1.5,-2+\i) {$s_i$};
			}
			\path (s1) --++ (0,0.6) node[midway,scale=1.5] {$\vdots$};
			\path (s2) --++ (0,0.6) node[midway,scale=1.5] {$\vdots$};
			\foreach \i in {3} {
				\node[place] (s\i) at (-1.5,-2+\i) {$s_I$};
			}
            % t=3
			\foreach \i in {2} {
				\node[place] (v\i) at (2.5,-0.5+\i) {$s_u$};
			}
			\path (v1) --++ (0,1) node[midway,scale=1.5] {$\vdots$};
			\path (v2) --++ (0,0.6) node[midway,scale=1.5] {$\vdots$};
                \path (e1) --++ (0,1.5) node[midway,scale=1.5] {$\vdots$};
                \path (e1) --++ (1,0) node[midway,scale=1.5] {$\cdots$};
			\foreach \i in {3} {
				\node[place] (v\i) at (2.5,-0.5+\i) {$s_{v}$};
			}
			\node[transition] at (-3,0) {$ a_0 $}  edge [pre] (s0) edge [post]node[fill=white]{$ \varepsilon $} (s1) edge [post] node[fill=white]{$ \varepsilon $} (s2) edge [post] node[fill=white]{$ \varepsilon $} (s3) edge [post, in=180, out=70] node [left,fill=white,pos=0.6]{$\frac{1-\varepsilon|I|}{|V|}$} (v3) edge [post, in=170, out=65] node [left,fill=white,pos=0.55]{$\frac{1-\varepsilon|I|}{|V|}$} (v2); 
			% t=2
			\foreach \i in {2} {
				\node[place, label=below:{$e\in E[V_{i,j}]$}] (e\i) at (0.5,-1+\i) {$s_e$};
			}
			\path (s1) --++ (0,0.6) node[midway,scale=1.5] {$\vdots$};
			\path (s2) --++ (0,0.6) node[midway,scale=1.5] {$\vdots$};
	     	\node[transition] at (-0.5,1) {$ a_{|J|} $}  edge [pre] (s2);
	     	\node[transition] at (-0.5,-1) {$ a_1 $}  edge [pre] (s2);
			\node[transition] at (-0.5,0) {$ a_j $}  edge [pre] (s2) edge [post]node[fill=white]{$ \frac{1}{3\bar{m}} $} (e1) edge [post] node[fill=white]{$  \frac{2}{3\bar{m}} $} (e2) ; 
			\path (s2) --++ (1.8,0.7) node[midway,scale=1.5] {$\vdots$};
			\path (s2) --++ (1.8,-0.85) node[midway,scale=1.5] {$\vdots$};
			\node[transition] at (1.5,1) {$ a_0 $}  edge [draw=black, pre] (e2) edge [draw=gray!40,post] (v1) edge [dashed,draw=gray!40,post] (v2) edge [dashed, draw=gray!40,post]  (v3) ; 
		% t=5
            \draw (v1) edge [post] (t1);
		\node[transition]  at (3.5,1.5) {$ a_0 $} edge [pre] (v2)  edge [draw=gray!40,post] (t1) edge [dashed,draw=gray!40,post](t2);
		\node[transition]  at (3.5,2.5) {$ a_0 $} edge [pre] (v3) edge [draw=gray!40,post](t1) edge [dashed, draw=gray!40,post] (t2);
		\end{tikzpicture}}
		\caption{Construction of the Reduction from max-min VC to {\transitionP}. The solid gray arrows represent nominal deterministic outcomes of a transition, with the dashed gray lines representing the alternative transition outcomes from the uncertainty set.}
		\medskip    
		\label{fig:sigma-2-p}	 
	\end{figure}

The constructed MDP is shown in \cref{fig:sigma-2-p}.
Note that the only states with multiple actions available are the states $s_i$ for $i \in I$, each having the action set $A(s_i) = \{a_j : j \in J\}$. 
Thus, there is a one-to-one correspondence between policies $\pi$ in the described MDP and functions $t : I \rightarrow J$, defined by the relation $t(i) = j$ if and only if $\pi(s_i) = a_j$.
Given a policy $\pi$, let  $ t^{\pi} $ denote the corresponding function. 
We complete the reduction by showing that for any policy $\pi$, the worst-case reward obtained by $\pi$ is strictly larger than  $1 - \frac{1 - \varepsilon|I|}{|V|} (\ell - 1) - \frac{2\varepsilon}{\bar{m}} m$, if and only if the size of a minimum vertex cover in the graph $G_{t^\pi}$ is at least $\ell$.
  
  Given a policy $\pi$, consider a worst-case scenario $\hat{z}$.
  For each $e \in E$, let $p_e$ denote the probability that $s_e$ is reached under $\pi$ (note that $\hat{z}$ does not influence this probability, as the transition probabilities of states before $s_e$ are not subjected to uncertainty).
  For each $v \in V$, let $p_v$ be the probability that state $s_v$ is reached under $\pi$ directly from $s_0$, i.e., without having visited any state $s_e$ for $e \in E$ before (again note that this probability does not depend on $\hat{z}$).
  Note that that $p_v = \frac{1 - \varepsilon|I|}{|V|}$ for each $v \in V$ and that $p_e = \frac{2\varepsilon}{3\bar{m}}$ for $e \in E_{t^\pi}$, and that $p_e \leq \frac{\varepsilon}{3\bar{m}}$ for $e \in E \setminus E_{t^\pi}$. In particular, by choice of $\varepsilon$ (which is the unique solution to $\frac{1 - |I|\varepsilon}{|V|} = \frac{\varepsilon}{2\bar{m}}$), we have $p_e > p_v$ if $e \in E_{t^\pi}$ and $p_e < p_v$ otherwise.

  Let $\hat{V} := \{v \in V : \hat{z}(s_v^r, a_0) = t_0\}$ and $\hat{E} := \{e \in E : \hat{z}(s_e, a_0) = s_v \text{ for some } v \in \hat{V}\}$.
  Note that the expected reward obtained by $\pi$ under $\hat{z}$ is equal to the probability of reaching state $t_1$, which is equal to
  $1 - \sum_{v \in \hat{V}} p_v - \sum_{e \in \hat{E}} p_e$,
  because the only possibility for the process to reach state $t_0$ is via some $s_v$ for $v \in \hat{V}$. 
  From the above bounds on $p_e$ and $p_v$, we deduce that 
  \begin{align}
      1 - \sum_{v \in \hat{V}} p_v - \sum_{e \in \hat{E}} p_e \;\geq\; 1 - (\ell - 1) \cdot \frac{1 - |I|\varepsilon}{|V|} - m \cdot \frac{2\varepsilon}{3\bar{m}}, \label{eq:sigma-2-p-reward-bound}
  \end{align}
  because $|\hat{E} \cup \hat{V}| \leq k = m + \ell - 1$ and there are only $|E_{t^\pi}| = m$ edges $e$ with $p_e = \frac{2\varepsilon}{3\bar{m}}$.
  Furthermore, \eqref{eq:sigma-2-p-reward-bound} holds with equality if and only if $\hat{E} = E_{t^{\pi}}$ and $|\hat{V}| = \ell - 1$, in which case $\hat{V} \subseteq V$ must cover all edges in $E_{t^{\pi}}$, i.e., $\hat{V} \cap V_{t^\pi}$ is a vertex cover in $G_{t^\pi}$ of size at most $\ell - 1$. Thus the worst-case reward of $\pi$ equals $1 - \sum_{v \in \hat{V}} p_v - \sum_{e \in \hat{E}} p_e$ if and only if there exists a vertex cover of size less than $\ell$ in $G_{t^\pi}$, and it is strictly larger otherwise.
\end{proof}

%\subsection{Inapproximability}
\label{sec:transition-inapproximability}

Finally, we also show that {\transitionP} does not allow for any approximation algorithm by proving \cref{thm:transition-inapproximability}.

\thmTransitionInapproximability*

\begin{proof}
% \todo[inline]{Make this proof consistent with the notation and style in the rest of the paper.}

We will prove the theorem by reduction from \textsc{3-SAT} to {\transitionP} with an uncertainty budget $k = 2$. In this problem we are given $n$ Boolean variables $x_j$, $j\in N :=\{1,\dots,n\}$ and $m$ clauses $C_j$ for $j\in M :=\{1,\dots,m\}$, with each clause consisting of the disjunction of three literals of the variables. %For any integer $k$, we use [k] to denote the set of integers between 1 and $k$.
Let $L := \{x_j, \neg x_j : j \in N\}$ be the set of literals on the variables. We identify each clause with the set of its three literals (e.g., $C_5 = \{x_4, \neg x_3, x_7\}$). A \emph{truth assignment} is a subset $A \subseteq L$ of the literals containing exactly one literal for each variable, i.e., $|\{x_j, \neg x_j\} \cap A| = 1$ for each $j \in  N$. A clause $C_i$ is fulfilled by truth assignment $A$ if $C_i \cap A \neq \emptyset$, i.e., if at least one of the literals in the clause is selected in the truth assignment.
    The task is to decide whether there exists a truth assignment that fulfills all clauses.

    Given an instance of \textsc{3-SAT}, we construct an instance of {\transitionP} as follows. We introduce the following states:
    \begin{align*}
        S & = \{s_0, s_{c_{m+1}}, s_{n+1}, d, t_1, t_2\} \cup \{s_{c_i} : i \in M\} \cup \{s_j, s_{x_j}, s_{\neg x_j} : j \in N\}
    \end{align*}
    The action sets of the states are the following:
    \begin{alignat*}{2}
        A(s_{c_i}) & = \{a_y : y \in C_i\} \quad && \text{for } i \in M\\
        A(s_{j}) & = \{a_{x_j}, a_{\neg x_j}\} && \text{for } j \in N\\
        A(s_{y}) & = \{a, a'\} && \text{for } y \in L
    \end{alignat*}
For all the other states, $ A(s) = \{a\} $. Terminal rewards are deterministic and given by $r(t_1) = 1$ and $r(t_2) = 0$. The transition probabilities are given as follows. Action $a$ in the starting state $s_0$ is the only action with randomized outcome, with probabilities
    \begin{align*}
        p(s_{c_1} | s_0, a) = p(s'_{c_1} | s_0, a) = \tfrac{1}{2}.
    \end{align*}
    \iffalse
    The following transitions are deterministic, that is
    \begin{align*} 
    &p_{s'_{c_1},s_1}(a)=p_{l_1, b_1}(a)=p_{s'_j,t_2}(a)=1 \quad && \text{for } j \in N\quad  &&\\
    &p_{s_y,d_2} (a') = 1 \quad && \text{for } y \in L\\
    &p_{s^y_{c_i},s_y}(a)=1 \quad &&\text{for } y\in C_i, i\in M
    \end{align*}
    \fi
    Furthermore, the following actions are entirely deterministic:
    \begin{alignat*}{2}
        z(s_i, a_y) & = s_{y} \quad && \text{for } i \in N,\; y \in \{x_i, \neg x_i\}\\
        z(s_y, a') & = d \quad && \text{for } y \in L
    \end{alignat*}
    All other actions are uncertain deterministic with the following outcomes:
    \begin{alignat*}{3}
        z(s_{c_i}, a_{y}) & = s_{c_{i+1}}, & \quad  Z(s_{c_i}, a_{y}) & = \{s_y\} \quad & \text{for } i \in M, y \in C_i\\
        z(s_{y}, a) & = s_{j+1}, & Z(s_{y}, a) & = \{t_2\} & \text{for } j \in N, y \in \{x_j, \neg x_j\}\\
        z(s_{c_{m+1}}, a) &= t_1, & Z(s_{c_{m+1}}, a) & = \{d\}   & \\
        z(s_{n+1}, a) &= t_1, & Z(s_{n+1}, a) & = \{d\}   & \\
        z(d, a) & = t_1, & Z(d, a) & = \{t_2\} &
    \end{alignat*}
    Finally, the number of actions whose transition probabilities can deviate is $k = 2$. 
    It is easy to verify that the digraph induced by the MDP is acyclic. See \cref{reduction2} for a depiction of the construction.
    	\begin{figure}[t] 
    	\centering 
    	\resizebox{1\columnwidth}{!}{
    		\begin{tikzpicture}[yscale=-1.6,xscale=1.5,thick,
    			every transition/.style={draw=red,fill=red!20,minimum size=4.8mm},
    			every place/.style={draw=blue,fill=blue!20, minimum size=6mm}]
    			every transition2/.style={draw=yellow,fill=yellow!20,minimum size=4.8mm},
    			\node[place,label=above:{$t=0$}] (s0) at (-7,0) {$s_0$};
    			\node[place,label=above:{}] (sc1) at (-5,-1) {$s_{c_1}$};
    			\node[place,label=above:{}] (sc2) at (-3,-1) {$s_{c_2}$};
    			\node[place,label=above:{}] (scm) at (-2,-1) {$s_{c_m}$};
    			\node[place,label=above:{}] (sm1) at (0.2,-1) {$s_{c_{m+1}}$};
    			%\node[place,label=above:{}] (b1) at (3.9,-1) {$b_1$};
    			\node[place,label=right:{}] (b2) at (3.9,2) {$s_{n+1}$}; %change the notation from b_2 to s_{n+1}
    			%\node[place,label=above:{$t=2n+m+2$}] (d1) at (5.9,-1) {$d_1$};
    			\node[place,label=right:{}] (d2) at (5.9,0.5) {$d$};% change the notation from b_2 to s_{n+1}
    			%\node[place,label=right:{}] (d3) at (5.9,3.2) {$d_3$};
    			\node[place,label=above:{$r(t_1)=1$}] (t1) at (7.9,-1) {$t_1$};
    			%\node[place,label=above:{}] (t1) at (10.9,1) {$t_1$};
    			\node[place,label=below:{$r(t_2)=0$}] (t2) at (7.9,2.5) {$t_2$};
    			% s1\label=above:{$t=1$}];
    			% insert m actions for each state
    			% t=1
    			%\node[place] (s1) at (-5,2) {$s'_{c_1}$};
    			\node[place] (sx1) at (-5,2) {$s_1$};
    			\node[transition] at (-6,0) {$a$}  edge [pre] (s0)  edge[post]node[above left]{$ \frac{1}{2}$} (sc1) edge[post]node[below left]{$ \frac{1}{2}$} (sx1);
    			% t=2
    			\node[place,label=above:{}] (sx11) at (-3.1,1.2) {$s_{x_1}$};
    			\node[place,label=above:{}] (sx10) at (-3.1,2.8) {$s_{\neg {x}_1}$};
    			%\node[place,label=above:{}] (sx12) at (-0.6,3.5) {$s'_1$};
    			\node[place,label=above:{}] (sx2) at (-1,2) {$s{_2}$};%-2.1
    			\node[place,label=above:{}] (sxn) at (0,2) {$s{_n}$};
    			\path (sx2) --++ (0.8,0) node[midway,scale=1.5] {$\cdots$};
    			\node[place,label=above:{}] (sxn1) at (1.9,1.2) {$s_{x_n}$};
    			\node[place,label=above:{}] (sxn0) at (1.9,2.8) {$s_{\neg{x}_n}$};
    			%\node[place,label=above:{}] (sxn2) at (3.9,3.5) {$s'_n$};
    			%\node[place,label=above:{}] (sxn3) at (8,3.5) {};
    			%\node[place] (s2) at (-1,1) {$s_2$};
    			\node[transition] at (-4.2,-1.8) {}  edge [pre] (sc1) edge[draw=gray!60, line width=1.5pt, post]node[above left]{} (sc2)
    			edge[dashed, draw=gray!60,line width=1.5pt, post] (-3.7,-0.8);
    			\node[transition] at (-4.2,-1) {$ a_{x_1} $}  edge [pre] (sc1)  edge[draw=gray!60, line width=1.5pt, post] (sc2) edge [dashed,draw=gray!60, line width=1.5pt, post]  (sx11);
    			\node[transition] at (-4.2,-0.2) {}  edge [pre] (sc1) edge[draw=gray!60, line width=1.5pt, post]node[above left]{} (sc2) edge[dashed, draw=gray!60,line width=1.5pt, post] (-3.8,0.6);
    			\node[transition] at (-1.2,-1.8) {}  edge [pre] (scm) edge[draw=gray!60,line width=1.5pt, post](sm1) edge[dashed, draw=gray!60,line width=1.5pt, post] (-0.7,-0.8);
    			\node[transition] at (-1.2,-1) {}  edge [pre] (scm) edge[draw=gray!60,line width=1.5pt, post]node[pos=0.4]{} (sm1) edge [dashed, draw=gray!60, line width=1.5pt, post] (-0.7,0) ;
    			\node[transition] at (-1.2,-0.2) {}  edge [pre] (scm) edge[draw=gray!60,line width=1.5pt, post]node[above left]{} (sm1)
    			edge[dashed, draw=gray!60,line width=1.5pt, post]node[above left]{} (-0.7,0.8);
    			\path (sc2) --++ (0.8,0) node[midway,scale=1.5] {$\cdots$};
    			\node[transition] at (-4.1,1.2) {$ a_{x_1} $}  edge [pre] (sx1)  edge[post](sx11);
    			%pos=[0,1]
    			\node[transition] at (-4.1,2.8) {$ a_{\neg x_1} $}  edge [pre] (sx1)  edge[post](sx10);
    			
    			%pos=[0,1]
    			\node[transition] at (-2.1,2) {$ a' $}  edge [pre] (sx10) edge[out=-30,in=210, post](d2);	
    			\node[transition] at (-2.1,2.8) {$ a $}  edge [pre] (sx10)  edge[draw=gray!60,line width=1.5pt, post](sx2) edge[out=25,in=165, dashed,line width=1.5pt, draw=gray!60, post](t2);	
    			\node[transition] at (-2.1,1.2) {$ a $}  edge [pre] (sx11)  edge[draw=gray!60,line width=1.5pt, post](sx2) edge[out=65,in=165, dashed, draw=gray!60,line width=1.5pt, post](t2);
    			\node[transition] at (-2.1,0.4) {$ a' $}  edge [pre] (sx11) edge[out=-10,in=210,post](d2); 		
    			\node[transition] at (0.9,1.2) {$ a_{x_n} $}  edge [pre] (sxn)  edge[post](sxn1);%3.4,2
    			\node[transition] at (0.9,2.8) {$ a_{\neg x_n} $}  edge [ pre] (sxn)  edge[post](sxn0);
    			\node[transition] at (2.9,2.8) {$ a $}  edge [pre] (sxn0)  edge[draw=gray!60,line width=1.5pt,post](b2) edge[dashed, draw=gray!60,line width=1.5pt, post](t2);
    			\node[transition] at (2.9,2) {$ a' $}  edge [pre] (sxn0) edge[post](d2) ; 
    			\node[transition] at (2.9,1.2) {$ a $}  edge [pre] (sxn1)  edge[draw=gray!60 ,line width=1.5pt,post](b2) edge[dashed, draw=gray!60,line width=1.5pt,post](t2);
    			\node[transition] at (2.9,0.4) {$ a' $}  edge [pre] (sxn1) edge[post](d2); 
    			\node[transition] at (4.9,2) {$ a $}  edge [pre] (b2)  edge[dashed,draw=gray!60, line width=1.5pt, post](d2) edge[draw=gray!60,line width=1.5pt, post](t1);
    			\node[transition] at (1.5,-1) {$ a $}  edge [pre] (sm1)  edge[draw=gray!60,line width=1.5pt, post] (t1) edge[dashed, draw=gray!60,line width=1.5pt, post](d2);
    			\node[transition] at (6.9, 0.5) {$ a $}  edge [pre] (d2)  edge[draw=gray!60,line width=1.5pt, post] (t1) edge[dashed, draw=gray!60,line width=1.5pt, post](t2);	
    		\end{tikzpicture}
    	}
    	\caption{Construction for the reduction from \textsc{3-SAT} to {\transitionP}. The solid gray arrows represent nominal deterministic outcomes of a transition, with the dashed gray lines representing the alternative transition outcomes from the uncertainty set. There are at most $k = 2$ actions for which the transitions can deviate from its nominal state.}
    	\label{reduction2}	 
    \end{figure}
    
    We show that there exists a policy $\pi$ with positive expected robust reward~$\hat{R}(\pi) > 0$ for the constructed instance of {\transitionP} if and only if there exists a satisfying truth assignment for the \textsc{3-SAT} instance. 
    Note that, by construction, the robust reward $\hat{R}(\pi)$ of a policy $\pi$ equals the probability of the process under the policy reaching state $t_1$ under any worst-case scenario.

    We first assume there exists a truth assignment $A$ fulfilling all clauses and we construct the following policy $\pi$, for which we will show $\hat{R}(\pi) > 0$. 
    For each $i \in M$, let $\pi(s_{c_i}) = a_y$ for some $y \in A\cap C_i$.
    For each $j \in N$, let 
    \begin{align*}
        \pi(s_j) = a_{x_j}, \quad \pi(s_{x_j}) = a, \quad \pi(s_{\neg x_j}) = a' \quad & \text{ if } \neg x_j \in A,\\
        \pi(s_j) = a_{\neg x_j}, \quad \pi(s_{x_j}) = a', \quad \pi(s_{\neg x_j}) = a \quad & \text{ if } x_j \in A.
    \end{align*}

    \begin{claim}\label{clm:assignment-to-policy}
        The policy $\pi$ constructed from the feasible truth assignment $A$ fulfils $\hat{R}(\pi) \geq \frac{1}{2}$.
    \end{claim}

    \begin{proof}
        Let $\hat{p} \in \hat{\mathcal{P}}$ be a worst-case scenario for policy $\pi$.
        Let $\hat{I}$ denote the set of state-action pairs whose transition kernel deviates from their nominal values in $\hat{p}$. 
        For $s, s' \in S$, we further let $p^{\pi}(s, s')$ and $\hat{p}^{\pi}(s,s')$, respectively, denote the conditional probability that, conditioning on the process under policy $\pi$ reaching state $s$, it will also reach state $s'$, for the nominal transition kernel $p$ and for the worst-case transition kernel $\hat{p}$, respectively.
        
        Note that $\hat{p}^{\pi}(s_0,s_{c_1}) = \hat{p}^{\pi}(s_0,s_1) = \frac{1}{2}$, as the transition kernel of $(s_0, a)$ does not allow for deviations.
        Therefore, $\hat{R}(\pi) = \frac{1}{2} \hat{p}^{\pi}(s_{c_1},t_1) + \frac{1}{2}\hat{p}^{\pi}(s_1,t_1)$.
        We show that $\hat{p}^{\pi}(s_1,t_1) = 1$ or $\hat{p}^{\pi}(s_{c_1},t_1) = 1$ and thus $\hat{R}(\pi) \geq \frac{1}{2}$.
    
        Note that under the nominal transition kernel, if $\pi$ reaches state $s_{c_i}$, it follows a deterministic path corresponding to state-action pairs $(s_{c_i}, a_y)$ with $y \in A \cap C_i$ and the final state-action pair $(s_{c_{m+1}}, a)$.
        If none of these state-action pairs is in $\hat{I}$, we have $\hat{p}^{\pi}(s_{c_1},t_1) = p^{\pi}(s_{c_1},t_1) = 1$ and we are done.
        Thus, assume that one of the aforementioned state-action pairs is in $\hat{I}$.
        We distinguish two cases:
        \begin{itemize}
            \item In case $(s_{c_i}, a_y) \in \hat{I}$ for some $i \in M$ and $y \in A \cap C_i$, then the action $a_y$ from state $s_{c_i}$ leads to $s_{y}$ under transition kernel $\hat{p}$. Because $y \in A$, we have $\pi(y) = a'$ and therefore $\hat{p}^{\pi}(s_{c_1}, d) = 1$.
            \item In case $(s_{c_{m+1}}, a) \in \hat{I}$, then this action leads to $d$ under the modified transition kernel $\hat{p}$ and again $\hat{p}^{\pi}(s_{c_1}, d) = 1$. 
        \end{itemize}
        Thus, in either case, $\hat{p}^{\pi}(s_{c_1}, d) = 1$.
        If $(d, a) \notin \hat{I}$, then $\hat{p}^{\pi}(s_{c_1}, t_1) = \hat{p}^{\pi}(d, t_1) = 1$ and we are done.
        If, on the other hand, $(d, a) \in \hat{I}$, then $(s_y, a) \notin \hat{I}$ for any $y \in L$ with $\neg y \in A$.
        Thus, by construction of $\pi$, we have $\hat{p}^{\pi}(s_1,t_1) = p^{\pi}(s_1,t_1) = 1$.
    \end{proof}
    
    Now, conversely, assume that there exists a policy $\pi$ with $\hat{R}(\pi)>0$.
    We establish the following claim:
    \begin{claim}
        Let $i \in M$ and let $y \in c_i$ with $\pi(s_{c_i}) = a_y$. 
        \begin{itemize}
            \item If $y = x_j$ for some $j \in N$, then 
        $\pi(s_j) = a_{\neg x_j}$ and $\pi(s_{\neg x_j}) = a$.
            \item If $y = \neg x_j$ for some $j \in N$, then 
        $\pi(s_j) = a_{x_j}$, and $\pi(s_{x_j}) = a$.
        \end{itemize}
    \end{claim}

    \begin{proof}
        As in the proof of \cref{clm:assignment-to-policy}, let $p^{\pi}(s, s')$ and $\hat{p}^{\pi}(s, s')$, respectively, denote the probability that conditioned on $\pi$ reaching state $s$, it will reach state $s'$ for the original transition kernel $p$ and a worst-case transition kernel $\hat{p}$.

        Note that, if $p^{\pi}(s_1, d) = 1$, then the alternative transition kernel $\hat{p}$ with \mbox{$\hat{p}(d | s_{c_{m+1}}, a) = 1$} and $\hat{p}(t_2 | d, a) = 1$ ensures that $\hat{p}^{\pi}(s_1, t_1) = \hat{p}^{\pi}(s_{c_1}, t_1) = 0$, and hence $\hat{R}(\pi) = 0$, a contradiction.
        Hence we must have $p^{\pi}(s_1, d) = 0$, which is only possible if $\pi(s_j) = s_y$ for $j \in N$ and $y \in \{x_j, \neg x_j\}$ implying $\pi(s_y) = a$.

        Now consider any $i \in M$ and let $y \in C_i$ be such that $\pi(s_{c_i}) = a_y$.
        Assume that $y = x_j$ for some $j \in N$ (the case that $y = \neg x_j$ is completely symmetric).
        If $\pi(s_j) = a_{x_j}$, then by the above, we must have 
        $p^{\pi}(s_1, s_{x_j}) = 1$
        and
        $\pi(s_{x_j}) = a$.
        But then the alternative transition kernel $\hat{p}$ with $\hat{p}(s_{x_j} | s_{c_i}, a_{x_j}) = 1$ and $\hat{p}(t_2 | s_{x_j}, a) = 1$ yields $\hat{R}(\pi) = 0$, a contradiction.
        Thus, $\pi(s_i) = \neg x_j$.
    \end{proof}
    
    Now define $A' := \{y \in L : \pi(s_{C_i}) = a_y \text{ for some } i \in N\}$.
    By the above claim, $|A' \cap \{x_j, \neg x_j\}| \leq 1$ for each $j \in N$.
    Moreover, $A' \cap C_i \neq \emptyset$ for $j \in M$, as $\pi(s_{C_i}) = a_y$ implies $y \in L$ by construction of the MDP.
    Therefore, $A := A' \cup \{x_j : j \in N, \neg x_j \notin A'\}$ is a feasible truth assignment for the given \textsc{3-SAT} instance.
\end{proof}

\section{Computational experiments} \label{sec: experiments}
Recall that the approximation algorithm consists of two subroutines: one based on a Knapsack Cover relaxation~(KC) and one based on a Generalized Assignment relaxation~(GA), as described in Section~\ref{sec:approximation}. The algorithm returns the better of the two solutions. We apply the approximation algorithm from Section~\ref{sec:approximation} on a set of test instances to evaluate its performance against the exact solutions obtained by solving the following integer program~(IP) with Gurobi:
\begin{align}
\max \quad & \sum_{s \in S_1} \sum_{a \in A(s)} c_{sa}\, x_{sa} - z
    \label{eq:ip-obj} \\
\text{s.t.} \quad
& \sum_{a \in A(s)} x_{sa} = 1
    & \forall\, s \in S_1 \label{eq:ip-assign} \\
& z \geq \sum_{s \in S_1} \sum_{a \in A(s)} d_{sa,t}\, x_{sa}
    & \forall\, t \in S_2 \label{eq:ip-robust} \\
& z \geq 0, \quad x_{sa} \in \{0,1\}  &\forall\, s \in S_1,\, a \in A(s)
    \label{eq:ip-domain}
\end{align}
where $c_{sa} = \sum_{t \in S_2} p(s \mid s_0, a_0)\, p(t \mid s,a)\, r(t)$ is the nominal reward contribution of action~$a$ at state~$s$, and $d_{sa,t} = p(s \mid s_0, a_0)\, p(t \mid s,a)\, r(t)$ is the potential loss at terminal state~$t$ (recall that $r'(t) = 0$ for all~$t \in S_2$). The auxiliary variable~$z$ captures the worst-case loss with constraints~\eqref{eq:ip-robust} ensuring that $z = \max_{t \in S_2} \sum_{s,a} d_{sa,t}\, x_{sa}$ at optimality. 

\paragraph{Implementation details} All algorithms are implemented in Python~3.12 with NumPy~2.4. The IP, all LP relaxations, and the minimum-cost matching in the Shmoys--Tardos rounding step are solved using Gurobi~13.0 with default settings in single-threaded mode. The matching subproblem is formulated as a linear program; the total unimodularity of the bipartite constraint matrix guarantees an integral optimal solution. All experiments were run on an Intel Core Ultra~9 285H machine (2.90\,GHz, 64\,GB RAM). We set the uncertainty budget to $k = 1$ throughout.

\subsection{Test instances} 

We generated a total of $93$~instances from three different types: instances based on a machine-replacement problem from the literature, instances based on the reduction from \textsc{3-Partition} in \cref{sec:reward}  (in both a uniformly random and a structurally hard variant), and high-impact instances designed so that the deviation of a single terminal reward has significant effect.

\paragraph{\textbf{Machine-replacement instances}} We adapt the machine replacement problem of \citet{wiesemann2013robust}, originally introduced by  \citet{delage2010percentile}, to our two-stage setting. The MDP models a single machine that has eight operative states $1, \ldots, 8$ and two repair states. The initial state is drawn uniformly over all 10 states. At each state, the decision maker chooses one of four actions: ``do nothing'' or one of three repair options, after which the machine transitions to a terminal state. Under ``do nothing'', transitions follow the nominal probabilities of the machine replacement model in \citet{wiesemann2013robust}. The three repair actions have transition probabilities drawn from a Dirichlet distribution concentrated on low-degradation states, with concentration parameter $M=20$. Terminal rewards are set as $r(t)=20-c(t)$ with costs $c(s_8)=20$, $c(R_1)=2$, $c(R_2)=10$, and $c(t)=0$ otherwise. Alternative rewards satisfy $r'(t)=0$ for all $t$. We also consider an alternative setting where the alternative reward is set to $r'(t)=r(t)/2$ instead of $0$.
To apply our algorithm in the alternative setting, we use the reduction described in Assumption 2 in \cref{sec:approximation} to obtain an equivalent instance (on a larger number of states) with $r'(t)=0$. We generated 10 instances for each of the two settings, resulting in a total of $20$ instances.

\paragraph{\textbf{\textsc{3-Partition} instances}} We construct instances directly from the $N\!P$-hardness reduction in the proof of Theorem~\ref{thm:2-stage-hardness}. For each $n \in \{5, 8, 10, 15, 20\}$, we generate $3n$ items with positive integer sizes $b_i$ summing to $nB$, yielding $|S_1| = 3n \in \{15, 24, 30, 45, 60\}$ intermediate states and $|S_2|=n$ terminal states. Each intermediate state $s_i$ has $n$ deterministic actions, with action $a_j$ leading the process to terminal state $t_j$ with probability $1$. The initial distribution is $p_{s_0}(s_i) = b_i/(nB)$, and rewards satisfy $r(t) = 1$, $r'(t) = 0$ for all $t$. While these instances are $N\!P$-hard in theory, their empirical difficulty depends on the distribution of item sizes~$b_i$. We consider two families of item-size vectors $(b_i)_{i=1}^{3n}$ that share the structure above but differ in their computational hardness. For each value of $n$ we generate $8$ hard instances and $3$ random instances using independent random seeds, resulting in a total of $55$ instances.

\begin{description}
\item[Uniformly random \textsc{3-Partition} instances] We independently sample $b_i \sim \mathcal{U}\{1, \ldots, 9\}$ for $i \in \{1, \dots, 3n\}$, set $B = \lceil (\sum_{i=1}^{3n} b_i)/n \rceil$, and add the residual $nB - \sum_{i=1}^{3n} b_i$ to $b_1$ to enforce $\sum_{i=1}^{3n} b_i = nB$ exactly. 

\item[Structured  \textsc{3-Partition} instances] To mimic the structure of the classic $N\!P$-hardness construction for \textsc{3-Partition}~\citep{garey1975complexity}, we also generate instances satisfying the condition $B/4 < b_i \leq B/2$. We fix $B = 100$ and draw $b_i \sim \mathcal{U}\{B/4 + 1, \ldots, B/2\}$,
%= \mathcal{U}\{26, \ldots, 50\}$, 
then perform a coordinate-wise random walk 
%within the feasible box $[26, 50]$ 
until $\sum_{i=1}^{3n} b_i = nB$. 
\end{description}

\paragraph{\textbf{High-impact instances}}
The \textsc{3-Partition} instances have a specific combinatorial structure (deterministic transitions, identical rewards).
To test the algorithm on instances with stochastic transitions where the deviation has a high impact, we consider instances with $|S_2| = 5$ fixed and transition probabilities drawn from a non-symmetric $\mathrm{Dirichlet}$ distribution that concentrates probability mass on two terminal states ($\alpha = 5.0$ for those two, $\alpha = 0.1$ for the rest). The Dirichlet distribution is a standard choice for
generating random probability vectors with controlled
concentration; the ratio $\alpha_{\max}/\alpha_{\min} = 50$ ensures that approximately $97\%$ of the transition
probability is directed to the two selected terminals. The two high-mass terminals are chosen with probability proportional to their reward~$r(t)$, so that high transition mass coincides with high reward and the single disruption has maximum impact. Terminal rewards are drawn from $\mathcal{U}[1,100]$ with $r'(t) = 0$, and the number of actions per state from $\mathcal{U}\{1,\ldots,5\}$.
We vary $|S_1| \in \{10, 20, 50, 100, 200, 300\}$ with 3 seeds per configuration, resulting in a total of $18$~instances.

\subsection{Results}

\paragraph{Solution quality}

Table~\ref{tab:approx-ratios} reports the approximation ratios across all three classes of instances, also indicating the performance of the two subroutines separately. The approximation ratio is defined as $\rho = \hat{R}(\pi_{\mathrm{approx}}) / \hat{R}(\pi^*)$. 
Across all configurations, the algorithm consistently achieves near-optimal performance, obtaining significantly more than $90\%$ of the optimal worst-case reward in all cases and thus substantially exceeding the theoretical worst-case guarantee of $1/(5+\varepsilon)$.   
\begin{table}[h!]
\captionsetup{justification=raggedright,singlelinecheck=false}
\centering
\begin{threeparttable}
\caption{{Approximation ratios (\%) on different instance sets.}}
\label{tab:approx-ratios}
\setlength{\tabcolsep}{6pt}
\begin{tabular}{
l
c
c
S[table-format=3.2]
S[table-format=3.2]
S[table-format=3.2]
S[table-format=3.2]
}
\toprule
\toprule
& & & \multicolumn{2}{c}{Subroutine (\%)}
        & \multicolumn{2}{c}{Approx. (\%)} \\
\cmidrule(lr){4-5}\cmidrule(lr){6-7}
Instance & $|S_1|$ & $|S_2|$ & {KC} & {GA} & {Average} & {Min} \\
\midrule
\addlinespace[0.6ex]
Machine-replacement ($r'(t)=0$)       & 10 & 10 & 100.00 & 97.29 & \textbf{100.00} & 100.00 \\
Machine-replacement ($r'(t)>0$)  & 10 & 10 & 99.98  & 92.72 & \textbf{99.99}  & 99.94 \\
\midrule
& 15 &  5 & 60.07 & 94.29 & \textbf{94.29} & 91.07 \\
& 24 &  8 & 54.62 & 95.18 & \textbf{95.18} & 93.75 \\
{\textbf{3-Partition (random)}}  & 30 & 10 & 54.89 & 95.78 & \textbf{95.78} & 95.42 \\
& 45 & 15 & 53.20 & 96.57 & \textbf{96.57} & 94.96 \\
& 60 & 20 & 51.97 & 97.27 & \textbf{97.27} & 96.90 \\
\addlinespace[0.5ex]
\midrule
& 15 &  5 & 60.00 & 93.59 & \textbf{93.59} & 91.00 \\
& 24 &  8 & 54.57 & 95.46 & \textbf{95.46} & 94.00 \\
 {\textbf{3-Partition (hard)}} & 30 & 10 & 54.89 & 96.82 & \textbf{96.82} & 95.11 \\
  & 45 & 15 & 52.93 & 97.89 & \textbf{97.89} & 96.79 \\
  & 60 & 20 & 51.84 & 98.32 & \textbf{98.32} & 97.74 \\
\addlinespace[0.5ex]
\midrule
    & 10  & 5 & 99.67 & 97.90 & \textbf{99.67} & 99.00 \\
    & 20  & 5 & 97.27 & 94.87 & \textbf{97.27} & 96.54 \\
    \textbf{High-impact}  & 50  & 5 & 97.62 & 93.67 & \textbf{97.62} & 93.34 \\
     & 100 & 5 & 94.86 & 93.84 & \textbf{95.52} & 92.04 \\
    & 200 & 5 & 95.59 & 93.55 & \textbf{95.74} & 94.68 \\
    & 300 & 5 & 94.30 & 89.72 & \textbf{94.30} & 91.13 \\
\bottomrule
\end{tabular}

\begin{tablenotes}[flushleft]
\footnotesize
\item \textit{Note.} 
Columns $|S_1|$ and $|S_2|$ indicate the number of intermediate and terminal states, respectively. Columns \emph{KC} and \emph{GA} report the approximation ratios of the two subroutines, Knapsack Cover and Generalized Assignment Problem, respectively, averaged over the instances with the described profile. 
\emph{Average} reports the approximation ratio of the entire algorithm, i.e., $\max\{\mathrm{KC}, \mathrm{GA}\}$, again averaged over the instances with the described profile. \emph{Min} reports the worst approximation ratio obtained by the algorithm among the instances with the described profile. 
% TODO: find better name than "random instances"
\end{tablenotes}

\end{threeparttable}
\end{table}

On the machine-replacement instances, the approximation ratio reaches $100\%$ when $r'(t)=0$ and remains above $99.94\%$ under general alternative rewards. 
On these instances, the KC subroutine dominates the GA subroutine: KC attains optimal solutions in most cases, while GA achieves average ratios of $97.29\%$ and $92.72\%$, respectively. On the \textsc{3-Partition} instances, across both variants, the combined approximation achieves a mean approximation ratio of $96.25\%$. The worst-case ratio across all $55$ instances is $91.00\%$, and the ratio improves monotonically with size in both variants, from approximately $ 94\%$ at $|S_1| = 15$ to over $97\%$ at $|S_1| = 60$. For these instances, the KC subroutine performs significantly worse, with ratios between $51.84\%$ and $60.07\%$, but this is compensated by the strong performance of the GA subroutine, so the combined output equals GA throughout, showing an interesting complementarity.  This behavior aligns with the theoretical analysis in \cref{sec:approximation}: the \textsc{3-Partition} construction yields instances where the loss $L(\pi^*)$ is small relative to $R(\pi^*)$, a regime in which the GA relaxation (\cref{thm:algo-gap}) provides stronger bounds. 

On the high-impact instances, the combined approximation algorithm attains a mean ratio of $96.68\%$ relative to the optimum ($\mathrm{min}~91.13\%$, $\mathrm{max}~100.00\%$). In contrast to the \textsc{3-Partition} instances, KC dominates on this class, with standalone ratios decreasing from $99.7\%$ at $|S_1| = 10$ to $94.3\%$ at $|S_1| = 300$.
Combining the two subroutines yields a measurable gain on individual seeds where GA outperforms KC at $|S_1| \in \{100, 200\}$. The worst-case ratio remains above $91\%$ across all configurations despite the reward-coupled adversarial construction. 

Overall, the complementarity between KC and GA is evident. KC performs better when the loss $L(\pi^*)$ constitutes a significant fraction of $R(\pi^*)$, while GA is more effective when the loss is relatively small. By combining both subroutines, the proposed algorithm achieves consistently high performance, with mean ratios above $93\%$ even on structurally hard \textsc{3-Partition} instances, far exceeding the theoretical worst-case guarantee of $1/(5+\varepsilon) \approx 19.6\%$, implying that the bound of \cref{thm:approximation} is very conservative in practice.

% Linear scale
\begin{figure}[t]
\centering
\begin{tikzpicture}
\begin{axis}[
    width=0.88\textwidth,
    height=7cm,
    xlabel={Problem size ($n$, $|S_1|=3n$)},
    ylabel={CPU time (s)},
    ymin=0, ymax=900,
    xmin=0, xmax=5.5,
    xtick={0.75, 1.75, 2.75, 3.75, 4.75},
    xticklabels={
        {$n{=}5$\\$|S_1|{=}15$},
        {$n{=}8$\\$|S_1|{=}24$},
        {$n{=}10$\\$|S_1|{=}30$},
        {$n{=}15$\\$|S_1|{=}45$},
        {$n{=}20$\\$|S_1|{=}60$}
    },
    x tick label style={align=center, font=\small},
    ylabel style={font=\small},
    y tick label style={font=\small},
    xlabel style={font=\small},
    ymajorgrids=true,
    grid style={dashed, gray!30},
    legend style={
        at={(0.03,0.97)},
        anchor=north west,
        font=\footnotesize,
        draw=gray!50,
        fill=white,
        fill opacity=0.9,
    },
    clip=false,
]

% Gurobi exact IP: converged (n=5, 8, 10)
\addplot[blue!70!black, thick, mark=*, mark size=2.5,
         mark options={fill=blue!50!white}]
    coordinates {(0.75, 0.36) (1.75, 0.48) (2.75, 30.3)};
\addlegendentry{Exact IP (converged)}

% Gurobi exact IP: time-limit (n=10, 15, 20)
\addplot[blue!70!black, thick, dashed, mark=o, mark size=3,
         mark options={solid, fill=white, line width=0.8pt}]
    coordinates {(2.75, 30.3) (3.75, 675.5) (4.75, 688.2)};
\addlegendentry{Exact IP ($\geq$ time limit)}

% KC subroutine
\addplot[green!60!black, thick, mark=triangle*, mark size=2.5,
         mark options={fill=green!40!white}]
    coordinates {(0.75, 0.043) (1.75, 0.262) (2.75, 0.752)
                 (3.75, 5.087) (4.75, 11.949)};
\addlegendentry{KC}

% GA subroutine
\addplot[red!60!black, thick, mark=square*, mark size=2.5,
         mark options={fill=red!40!white}]
    coordinates {(0.75, 0.013) (1.75, 0.023) (2.75, 0.028)
                 (3.75, 0.073) (4.75, 0.082)};
\addlegendentry{GA}

\end{axis}
\end{tikzpicture}
\caption{CPU time of the KC and GA subroutines versus the Gurobi IP solver on structurally hard \textsc{3-Partition} instances, averaged over $8$ seeds. Open markers with dashed line indicate sizes where Gurobi exceeds the $900$\,s time limit on $6$ of $8$ seeds. The combined algorithm runs KC and GA sequentially; its total time is dominated by KC.}
\label{fig:gap-scalability-hard}
\end{figure}

\paragraph{Computational performance on hard \textsc{3-Partition} instances}
We assess scalability on the $40$ structurally hard \textsc{3-Partition} instances, 
comparing the run-time of the combined approximation against the Gurobi IP solver.
On these instances, the branch-and-bound tree of Gurobi starts to grow rapidly starting from $n=10$ (corresponding to $|S_1| = 30$ intermediate states), hitting the imposed speed limit of $900$~seconds on $6$ of the $8$
seeds for $n = 15$ already.
The run-time of the combined approximation algorithms, however, remains low and can solve the respective problems within less than $15$~seconds throughout the instances.
Figure~\ref{fig:gap-scalability-hard} reports the
mean CPU time of both methods.

\section{Conclusion} \label{sec: conclusion}

In this paper, we discussed the complexity of computing optimal deterministic robust policies for Markov decision processes under the Lightning Does Not Strike Twice model.
Our results indicate that computing such deterministic policies is significantly harder than computing randomized policies.
In particular, the general problem {\transitionP} is $\Sigma_2^p$-hard, and (assuming $P \neq N\!P$) both {\transitionP} and {\rewardP} do not admit constant-factor approximation algorithms even under strong assumptions.
However, we were able to derive such an approximation algorithm for the interesting special case of {\twostageP}, in which the MDP consists only of two stages and only a single terminal reward deviates from its nominal value in any scenario. Numerical experiments confirm that the algorithm performs well in practice, consistently achieving approximation ratios above $91\%$, far exceeding the theoretical guarantee.

Our hardness results leave the possibility open for a constant-factor approximation to exist in the more general case of {\rewardP} with a constant number of terminal rewards deviating from their nominal value, but with no restrictions on the number of stages.
This variant of the problem generalizes \textsc{Vertex-Disjoint Paths} in a DAG with a constant number of terminal pairs, as shown by our reduction for proving \cref{thm:reward-inapproximability}.
Following the ideas exploited in \cref{sec:approximation}, it can also be seen that the problem captures the essential aspects of \textsc{Knapsack} and \textsc{Generalized Assignment}.
Thus, any constant-factor algorithm for {\rewardP} with a constant number of deviating terminal rewards would need to combine ideas capable of solving \textsc{Vertex-Disjoint Paths} and of approximating \textsc{Knapsack} and \textsc{Generalized Assignment}, making the question of the existence of such an algorithm particularly intriguing.

\bibliographystyle{elsarticle-num-names} 
\bibliography{references}

\begin{thebibliography}{36}
\expandafter\ifx\csname natexlab\endcsname\relax\def\natexlab#1{#1}\fi
\providecommand{\url}[1]{\texttt{#1}}
\providecommand{\href}[2]{#2}
\providecommand{\path}[1]{#1}
\providecommand{\DOIprefix}{doi:}
\providecommand{\ArXivprefix}{arXiv:}
\providecommand{\URLprefix}{URL: }
\providecommand{\Pubmedprefix}{pmid:}
\providecommand{\doi}[1]{\href{http://dx.doi.org/#1}{\path{#1}}}
\providecommand{\Pubmed}[1]{\href{pmid:#1}{\path{#1}}}
\providecommand{\bibinfo}[2]{#2}
\ifx\xfnm\relax \def\xfnm[#1]{\unskip,\space#1}\fi
%Type = Inproceedings
\bibitem[{Mannor et~al.(2012)Mannor, Mebel, and Xu}]{mannor2012lightning}
\bibinfo{author}{S.~Mannor}, \bibinfo{author}{O.~Mebel},
  \bibinfo{author}{H.~Xu},
\newblock \bibinfo{title}{Lightning does not strike twice: robust {MDPs} with
  coupled uncertainty},
\newblock in: \bibinfo{booktitle}{Proceedings of the 29th International
  Coference on International Conference on Machine Learning}, ICML'12,
  \bibinfo{publisher}{Omnipress}, \bibinfo{address}{Madison, WI, USA},
  \bibinfo{year}{2012}, p. \bibinfo{pages}{451–458}.
%Type = Article
\bibitem[{Mannor et~al.(2007)Mannor, Simester, Sun, and
  Tsitsiklis}]{mannor2007bias}
\bibinfo{author}{S.~Mannor}, \bibinfo{author}{D.~Simester},
  \bibinfo{author}{P.~Sun}, \bibinfo{author}{J.~N. Tsitsiklis},
\newblock \bibinfo{title}{Bias and variance approximation in value function
  estimates},
\newblock \bibinfo{journal}{Management Science} \bibinfo{volume}{53}
  (\bibinfo{year}{2007}) \bibinfo{pages}{308--322}.
  \DOIprefix\doi{10.1287/mnsc.1060.0614}.
%Type = Article
\bibitem[{Mannor et~al.(2016)Mannor, Mebel, and Xu}]{mannor2016robust}
\bibinfo{author}{S.~Mannor}, \bibinfo{author}{O.~Mebel},
  \bibinfo{author}{H.~Xu},
\newblock \bibinfo{title}{Robust {MDP}s with $k$-rectangular uncertainty},
\newblock \bibinfo{journal}{Mathematics of Operations Research}
  \bibinfo{volume}{41} (\bibinfo{year}{2016}) \bibinfo{pages}{1484--1509}.
  \DOIprefix\doi{10.1287/moor.2016.0786}.
%Type = Inproceedings
\bibitem[{Dolgov and Durfee(2005)}]{dolgov2005stationary}
\bibinfo{author}{D.~A. Dolgov}, \bibinfo{author}{E.~H. Durfee},
\newblock \bibinfo{title}{Stationary deterministic policies for constrained
  {MDP}s with multiple rewards, costs, and discount factors},
\newblock in: \bibinfo{booktitle}{Proceedings of the Nineteenth International
  Joint Conference on Artificial Intelligence (IJCAI-05)},
  \bibinfo{address}{Edinburgh, Scotland}, \bibinfo{year}{2005}, pp.
  \bibinfo{pages}{1326--1331}.
%Type = Article
\bibitem[{Nielsen and Kristensen(2006)}]{nielsen2006finding}
\bibinfo{author}{L.~R. Nielsen}, \bibinfo{author}{A.~R. Kristensen},
\newblock \bibinfo{title}{Finding the {$K$} best policies in a finite-horizon
  markov decision process},
\newblock \bibinfo{journal}{European Journal of Operational Research}
  \bibinfo{volume}{175} (\bibinfo{year}{2006}) \bibinfo{pages}{1164--1179}.
  \DOIprefix\doi{10.1016/j.ejor.2005.06.011}.
%Type = Article
\bibitem[{Ohlmann and Bean(2009)}]{ohlmann2009resource}
\bibinfo{author}{J.~W. Ohlmann}, \bibinfo{author}{J.~C. Bean},
\newblock \bibinfo{title}{Resource-constrained management of heterogeneous
  assets with stochastic deterioration},
\newblock \bibinfo{journal}{European Journal of Operational Research}
  \bibinfo{volume}{199} (\bibinfo{year}{2009}) \bibinfo{pages}{198--208}.
  \DOIprefix\doi{10.1016/j.ejor.2008.11.005}.
%Type = Article
\bibitem[{Ayvaci et~al.(2012)Ayvaci, Alagoz, and
  Burnside}]{ayvaci2012budgetary}
\bibinfo{author}{M.~U.~S. Ayvaci}, \bibinfo{author}{O.~Alagoz},
  \bibinfo{author}{E.~S. Burnside},
\newblock \bibinfo{title}{The effect of budgetary restrictions on breast cancer
  diagnostic decisions},
\newblock \bibinfo{journal}{Manufacturing \& Service Operations Management}
  \bibinfo{volume}{14} (\bibinfo{year}{2012}) \bibinfo{pages}{600--617}.
  \DOIprefix\doi{10.1287/msom.1110.0371}.
%Type = Article
\bibitem[{Garcia et~al.(2023)Garcia, Steimle, Marrero, and
  Sussman}]{garcia2023interpretable}
\bibinfo{author}{G.-G.~P. Garcia}, \bibinfo{author}{L.~N. Steimle},
  \bibinfo{author}{W.~J. Marrero}, \bibinfo{author}{J.~B. Sussman},
\newblock \bibinfo{title}{Interpretable policies and the price of
  interpretability in hypertension treatment planning},
\newblock \bibinfo{journal}{Manufacturing \& Service Operations Management}
  \bibinfo{volume}{26} (\bibinfo{year}{2023}) \bibinfo{pages}{80--94}.
  \DOIprefix\doi{10.1287/msom.2021.0373}.
%Type = Inproceedings
\bibitem[{Ying et~al.(2023)Ying, Zhang, Ding, Koppel, and
  Lavaei}]{ying2023scalable}
\bibinfo{author}{D.~Ying}, \bibinfo{author}{Y.~Zhang},
  \bibinfo{author}{Y.~Ding}, \bibinfo{author}{A.~Koppel},
  \bibinfo{author}{J.~Lavaei},
\newblock \bibinfo{title}{Scalable primal-dual actor-critic method for safe
  multi-agent {RL} with general utilities},
\newblock in: \bibinfo{booktitle}{Advances in Neural Information Processing
  Systems}, volume~\bibinfo{volume}{36}, \bibinfo{year}{2023}, pp.
  \bibinfo{pages}{36524--36539}. \DOIprefix\doi{10.52202/075280-1586}.
%Type = Inproceedings
\bibitem[{Santana et~al.(2016)Santana, Thi{\'e}baux, and
  Williams}]{santana2016rao}
\bibinfo{author}{P.~R. Q. e.~A. Santana}, \bibinfo{author}{S.~Thi{\'e}baux},
  \bibinfo{author}{B.~Williams},
\newblock \bibinfo{title}{{RAO}*: An algorithm for chance-constrained
  {POMDP}s},
\newblock in: \bibinfo{booktitle}{Proceedings of the Thirtieth AAAI Conference
  on Artificial Intelligence}, \bibinfo{publisher}{AAAI Press},
  \bibinfo{year}{2016}, pp. \bibinfo{pages}{3308--3314}.
  \DOIprefix\doi{10.1609/aaai.v30i1.10423}.
%Type = Article
\bibitem[{Nilim and El~Ghaoui(2005)}]{nilim2005robust}
\bibinfo{author}{A.~Nilim}, \bibinfo{author}{L.~El~Ghaoui},
\newblock \bibinfo{title}{Robust control of {Markov} decision processes with
  uncertain transition matrices},
\newblock \bibinfo{journal}{Operations Research} \bibinfo{volume}{53}
  (\bibinfo{year}{2005}) \bibinfo{pages}{780--798}.
  \DOIprefix\doi{10.1287/opre.1050.0216}.
%Type = Article
\bibitem[{Iyengar(2005)}]{iyengar2005robust}
\bibinfo{author}{G.~N. Iyengar},
\newblock \bibinfo{title}{Robust dynamic programming},
\newblock \bibinfo{journal}{Mathematics of Operations Research}
  \bibinfo{volume}{30} (\bibinfo{year}{2005}) \bibinfo{pages}{257--280}.
  \DOIprefix\doi{10.1287/moor.1040.0129}.
%Type = Article
\bibitem[{Satia and Lave~Jr(1973)}]{satia1973markovian}
\bibinfo{author}{J.~K. Satia}, \bibinfo{author}{R.~E. Lave~Jr},
\newblock \bibinfo{title}{Markovian decision processes with uncertain
  transition probabilities},
\newblock \bibinfo{journal}{Operations Research} \bibinfo{volume}{21}
  (\bibinfo{year}{1973}) \bibinfo{pages}{728--740}.
  \DOIprefix\doi{10.1287/opre.21.3.728}.
%Type = Article
\bibitem[{White~III and Eldeib(1994)}]{white1994markov}
\bibinfo{author}{C.~C. White~III}, \bibinfo{author}{H.~K. Eldeib},
\newblock \bibinfo{title}{Markov decision processes with imprecise transition
  probabilities},
\newblock \bibinfo{journal}{Operations Research} \bibinfo{volume}{42}
  (\bibinfo{year}{1994}) \bibinfo{pages}{739--749}.
  \DOIprefix\doi{10.1287/opre.42.4.739}.
%Type = Article
\bibitem[{Givan et~al.(2000)Givan, Leach, and Dean}]{givan2000bounded}
\bibinfo{author}{R.~Givan}, \bibinfo{author}{S.~Leach},
  \bibinfo{author}{T.~Dean},
\newblock \bibinfo{title}{Bounded-parameter {Markov} decision processes},
\newblock \bibinfo{journal}{Artificial Intelligence} \bibinfo{volume}{122}
  (\bibinfo{year}{2000}) \bibinfo{pages}{71--109}.
  \DOIprefix\doi{10.1016/s0004-3702(00)00047-3}.
%Type = Article
\bibitem[{Goyal and Grand-Clement(2023)}]{goyal2023robust}
\bibinfo{author}{V.~Goyal}, \bibinfo{author}{J.~Grand-Clement},
\newblock \bibinfo{title}{Robust {Markov} decision processes: Beyond
  rectangularity},
\newblock \bibinfo{journal}{Mathematics of Operations Research}
  \bibinfo{volume}{48} (\bibinfo{year}{2023}) \bibinfo{pages}{203--226}.
  \DOIprefix\doi{10.1287/moor.2022.1259}.
%Type = Book
\bibitem[{Puterman(1994)}]{puterman1994markov}
\bibinfo{author}{M.~L. Puterman}, \bibinfo{title}{Markov decision processes:
  discrete stochastic dynamic programming}, \bibinfo{publisher}{John Wiley \&
  Sons}, \bibinfo{address}{Hoboken, New Jersey}, \bibinfo{year}{1994}.
  \DOIprefix\doi{10.1002/9780470316887}.
%Type = Article
\bibitem[{Wiesemann et~al.(2013)Wiesemann, Kuhn, and
  Rustem}]{wiesemann2013robust}
\bibinfo{author}{W.~Wiesemann}, \bibinfo{author}{D.~Kuhn},
  \bibinfo{author}{B.~Rustem},
\newblock \bibinfo{title}{Robust {Markov} decision processes},
\newblock \bibinfo{journal}{Mathematics of Operations Research}
  \bibinfo{volume}{38} (\bibinfo{year}{2013}) \bibinfo{pages}{153--183}.
  \DOIprefix\doi{10.1287/moor.1120.0566}.
%Type = Article
\bibitem[{Li and Shapiro(2025)}]{li2025rectangularity}
\bibinfo{author}{Y.~Li}, \bibinfo{author}{A.~Shapiro},
\newblock \bibinfo{title}{Rectangularity and duality of distributionally robust
  {Markov} decision processes},
\newblock \bibinfo{journal}{Mathematical Programming}  (\bibinfo{year}{2025})
  \bibinfo{pages}{1--42}.
%Type = Article
\bibitem[{Wang et~al.(2024)Wang, Gao, and Zha}]{wang2024reliable}
\bibinfo{author}{J.~Wang}, \bibinfo{author}{R.~Gao}, \bibinfo{author}{H.~Zha},
\newblock \bibinfo{title}{Reliable off-policy evaluation for reinforcement
  learning},
\newblock \bibinfo{journal}{Operations Research} \bibinfo{volume}{72}
  (\bibinfo{year}{2024}) \bibinfo{pages}{699--716}.
%Type = Article
\bibitem[{Bertsimas and Sim(2004)}]{bertsimas2004price}
\bibinfo{author}{D.~Bertsimas}, \bibinfo{author}{M.~Sim},
\newblock \bibinfo{title}{The price of robustness},
\newblock \bibinfo{journal}{Operations Research} \bibinfo{volume}{52}
  (\bibinfo{year}{2004}) \bibinfo{pages}{35--53}.
  \DOIprefix\doi{10.1287/opre.1030.0065}.
%Type = Inproceedings
\bibitem[{Bansal and Sviridenko(2006)}]{bansal2006santa}
\bibinfo{author}{N.~Bansal}, \bibinfo{author}{M.~Sviridenko},
\newblock \bibinfo{title}{The santa claus problem},
\newblock in: \bibinfo{booktitle}{Proceedings of the thirty-eighth annual ACM
  symposium on Theory of computing}, \bibinfo{year}{2006}, pp.
  \bibinfo{pages}{31--40}.
%Type = Article
\bibitem[{Woeginger(2021)}]{woeginger2021trouble}
\bibinfo{author}{G.~J. Woeginger},
\newblock \bibinfo{title}{The trouble with the second quantifier},
\newblock \bibinfo{journal}{4OR} \bibinfo{volume}{19} (\bibinfo{year}{2021})
  \bibinfo{pages}{157--181}. \DOIprefix\doi{10.1007/s10288-021-00477-y}.
%Type = Inproceedings
\bibitem[{Bonet et~al.(2009)Bonet, Palacios, and Geffner}]{bonet2009automatic}
\bibinfo{author}{B.~Bonet}, \bibinfo{author}{H.~Palacios},
  \bibinfo{author}{H.~Geffner},
\newblock \bibinfo{title}{Automatic derivation of memoryless policies and
  finite-state controllers using classical planners},
\newblock in: \bibinfo{booktitle}{Proceedings of the Nineteenth International
  Conference on Automated Planning and Scheduling}, \bibinfo{publisher}{AAAI
  Press}, \bibinfo{year}{2009}, pp. \bibinfo{pages}{34--41}.
  \DOIprefix\doi{10.1609/icaps.v19i1.13379}.
%Type = Inproceedings
\bibitem[{Kesselheim et~al.(2023)Kesselheim, Molinaro, and
  Singla}]{kesselheim2023online}
\bibinfo{author}{T.~Kesselheim}, \bibinfo{author}{M.~Molinaro},
  \bibinfo{author}{S.~Singla},
\newblock \bibinfo{title}{Online and bandit algorithms beyond $\ell_p$ norms},
\newblock in: \bibinfo{booktitle}{Proceedings of the 2023 Annual ACM-SIAM
  Symposium on Discrete Algorithms (SODA)}, \bibinfo{organization}{SIAM},
  \bibinfo{year}{2023}, pp. \bibinfo{pages}{1566--1593}.
%Type = Article
\bibitem[{Eberle et~al.(2025)Eberle, Gupta, Megow, Moseley, and
  Zhou}]{eberle2025configuration}
\bibinfo{author}{F.~Eberle}, \bibinfo{author}{A.~Gupta},
  \bibinfo{author}{N.~Megow}, \bibinfo{author}{B.~Moseley},
  \bibinfo{author}{R.~Zhou},
\newblock \bibinfo{title}{Configuration balancing for stochastic requests: F.
  eberle et al.},
\newblock \bibinfo{journal}{Mathematical Programming} \bibinfo{volume}{210}
  (\bibinfo{year}{2025}) \bibinfo{pages}{243--279}.
%Type = Article
\bibitem[{Zhang et~al.(2017)Zhang, Sundaramoorthy, Grossmann, and
  Pinto}]{zhang2017multiscale}
\bibinfo{author}{Q.~Zhang}, \bibinfo{author}{A.~Sundaramoorthy},
  \bibinfo{author}{I.~E. Grossmann}, \bibinfo{author}{J.~M. Pinto},
\newblock \bibinfo{title}{Multiscale production routing in multicommodity
  supply chains with complex production facilities},
\newblock \bibinfo{journal}{Computers \& Operations Research}
  \bibinfo{volume}{79} (\bibinfo{year}{2017}) \bibinfo{pages}{207--222}.
  \DOIprefix\doi{10.1016/j.cor.2016.11.001}.
%Type = Article
\bibitem[{Long et~al.(2024)Long, Wang, Zhai, Liang, Jiang, and
  Zhao}]{long2024data}
\bibinfo{author}{J.~Long}, \bibinfo{author}{N.~Wang},
  \bibinfo{author}{J.~Zhai}, \bibinfo{author}{C.~Liang},
  \bibinfo{author}{S.~Jiang}, \bibinfo{author}{L.~Zhao},
\newblock \bibinfo{title}{Data driven multi-objective economic-environmental
  robust optimization for refinery planning with multiple modes under
  uncertainty},
\newblock \bibinfo{journal}{Computers \& Industrial Engineering}
  \bibinfo{volume}{198} (\bibinfo{year}{2024}) \bibinfo{pages}{110697}.
  \DOIprefix\doi{10.1016/j.cie.2024.110697}.
%Type = Article
\bibitem[{de~Oliveira et~al.(2018)de~Oliveira, Ribeiro, and
  Cicogna}]{oliveira2018uncertainty}
\bibinfo{author}{S.~M. de~Oliveira}, \bibinfo{author}{C.~d.~O. Ribeiro},
  \bibinfo{author}{M.~P.~V. Cicogna},
\newblock \bibinfo{title}{Uncertainty effects on production mix and on hedging
  decisions: The case of brazilian ethanol and sugar},
\newblock \bibinfo{journal}{Energy Economics} \bibinfo{volume}{70}
  (\bibinfo{year}{2018}) \bibinfo{pages}{516--524}.
  \DOIprefix\doi{10.1016/j.eneco.2018.01.025}.
%Type = Article
\bibitem[{Goh et~al.(2018)Goh, Bayati, Zenios, Singh, and Moore}]{goh2018data}
\bibinfo{author}{J.~Goh}, \bibinfo{author}{M.~Bayati}, \bibinfo{author}{S.~A.
  Zenios}, \bibinfo{author}{S.~Singh}, \bibinfo{author}{D.~Moore},
\newblock \bibinfo{title}{Data uncertainty in markov chains: Application to
  cost-effectiveness analyses of medical innovations},
\newblock \bibinfo{journal}{Operations Research} \bibinfo{volume}{66}
  (\bibinfo{year}{2018}) \bibinfo{pages}{697--715}.
%Type = Article
\bibitem[{Garey and Johnson(1975)}]{garey1975complexity}
\bibinfo{author}{M.~R. Garey}, \bibinfo{author}{D.~S. Johnson},
\newblock \bibinfo{title}{Complexity results for multiprocessor scheduling
  under resource constraints},
\newblock \bibinfo{journal}{SIAM journal on Computing} \bibinfo{volume}{4}
  (\bibinfo{year}{1975}) \bibinfo{pages}{397--411}.
%Type = Article
\bibitem[{Even et~al.(1976)Even, Itai, and Shamir}]{even1976complexity}
\bibinfo{author}{S.~Even}, \bibinfo{author}{A.~Itai},
  \bibinfo{author}{A.~Shamir},
\newblock \bibinfo{title}{On the complexity of timetable and multicommodity
  flow problems},
\newblock \bibinfo{journal}{SIAM Journal on Computing} \bibinfo{volume}{5}
  (\bibinfo{year}{1976}) \bibinfo{pages}{691--703}.
  \DOIprefix\doi{10.1137/0205048}.
%Type = Article
\bibitem[{Pruhs and Woeginger(2007)}]{pruhs2007approximation}
\bibinfo{author}{K.~Pruhs}, \bibinfo{author}{G.~J. Woeginger},
\newblock \bibinfo{title}{Approximation schemes for a class of subset selection
  problems},
\newblock \bibinfo{journal}{Theoretical Computer Science} \bibinfo{volume}{382}
  (\bibinfo{year}{2007}) \bibinfo{pages}{151--156}.
  \DOIprefix\doi{10.1016/j.tcs.2007.03.006}.
%Type = Article
\bibitem[{Shmoys and Tardos(1993)}]{shmoys1993approximation}
\bibinfo{author}{D.~B. Shmoys}, \bibinfo{author}{{\'E}.~Tardos},
\newblock \bibinfo{title}{An approximation algorithm for the generalized
  assignment problem},
\newblock \bibinfo{journal}{Mathematical Programming} \bibinfo{volume}{62}
  (\bibinfo{year}{1993}) \bibinfo{pages}{461--474}.
  \DOIprefix\doi{10.1007/bf01585178}.
%Type = Incollection
\bibitem[{Ko and Lin(1995)}]{ko1995complexity}
\bibinfo{author}{K.-I. Ko}, \bibinfo{author}{C.-L. Lin},
\newblock \bibinfo{title}{On the complexity of min-max optimization problems
  and their approximation},
\newblock in: \bibinfo{booktitle}{Minimax and Applications},
  \bibinfo{publisher}{Springer}, \bibinfo{address}{Boston, MA},
  \bibinfo{year}{1995}, pp. \bibinfo{pages}{219--239}.
%Type = Article
\bibitem[{Delage and Mannor(2010)}]{delage2010percentile}
\bibinfo{author}{E.~Delage}, \bibinfo{author}{S.~Mannor},
\newblock \bibinfo{title}{Percentile optimization for {Markov} decision
  processes with parameter uncertainty},
\newblock \bibinfo{journal}{Operations research} \bibinfo{volume}{58}
  (\bibinfo{year}{2010}) \bibinfo{pages}{203--213}.

\end{thebibliography}

\end{document}